\documentclass[12pt]{article}

\usepackage{epsfig,color,amsmath,amssymb,graphics,verbatim,amsfonts,psfrag,mathrsfs}
\usepackage[letterpaper,margin=0.9in]{geometry}

\usepackage{subcaption}

\usepackage[colorinlistoftodos]{todonotes} 

\usepackage{amsthm}
\newtheorem{thm}{Theorem}[section]

\theoremstyle{remark}
\newtheorem{remark}[thm]{Remark}

\def\txtd{{\textnormal{d}}}
\def\txte{{\textnormal{e}}}

\def\txtD{{\textnormal{D}}}

\def\Re{{\textnormal{Re}}}

\def\I{\infty}

\def\R{\mathbb{R}}
\def\C{\mathbb{C}}
\def\N{\mathbb{N}}

\def\I{\infty}

\newcommand{\be}{\begin{equation}}
\newcommand{\ee}{\end{equation}}
\newcommand{\bea}{\begin{eqnarray}}
\newcommand{\eea}{\end{eqnarray}}
\newcommand{\beann}{\begin{eqnarray*}}
\newcommand{\eeann}{\end{eqnarray*}}
\newcommand{\benn}{\begin{equation*}}
\newcommand{\eenn}{\end{equation*}}

\def\I{\infty}

\newcommand{\cM}{{\mathcal M}}  
\newcommand{\cO}{{\mathcal O}}  
\newcommand{\cS}{{\mathcal S}}  
\newcommand{\cT}{{\mathcal T}}  
\newcommand{\cY}{{\mathcal Y}}  

\newcommand{\UU}{\mathcal{U}}
\newcommand{\VV}{\mathcal{V}}
\newcommand{\trace}{\mathrm{tr}\,}

\newcommand{\eps}{\varepsilon}
\newcommand{\ueps}{u_\eps}
\newcommand{\Teps}{T_\eps}
\DeclareMathOperator{\sgn}{sgn}

\iftrue 
\usepackage{tikz}
\usepackage{pgfplots}
\usetikzlibrary{pgfplots.groupplots}

\newcommand{\includeTikzOrEps}[1]{\tikzexternalenable 
  \tikzsetnextfilename{#1_img}
  {\include{#1}} \tikzexternaldisable} 
\usetikzlibrary{external}
\tikzexternaldisable
 
\pgfplotsset{
  log x ticks with fixed point/.style={
      xticklabel={
        \pgfkeys{/pgf/fpu=true}
        \pgfmathparse{exp(\tick)}%
        \pgfmathprintnumber[fixed relative, precision=3]{\pgfmathresult}
        \pgfkeys{/pgf/fpu=false}
      }
  },
  log y ticks with fixed point/.style={
      yticklabel={
        \pgfkeys{/pgf/fpu=true}
        \pgfmathparse{exp(\tick)}%
        \pgfmathprintnumber[fixed relative, precision=3]{\pgfmathresult}
        \pgfkeys{/pgf/fpu=false}
      }
  }
}

\pgfplotsset{
    discard if not/.style 2 args={
        x filter/.code={
            \edef\tempa{\thisrow{#1}}
            \edef\tempb{#2}
            \ifx\tempa\tempb
            \else
                \def\pgfmathresult{inf}
            \fi
        }
    }
}

\else
\newcommand{\includeTikzOrEps}[1]{\includegraphics{#1_img}}
\fi

\begin{document}

\author{Franz Achleitner\thanks{Technische Universit\"at Wien, Institute for Analysis and Scientific Computing, Wiedner Hauptstraße 8-10 1040 Wien, Austria},~~Christian Kuehn\thanks{Technical University of Munich (TUM), Faculty of Mathematics, 85748 Garching bei M\"unchen, Germany},~~Jens M.~Melenk\footnotemark[1]~~and~~Alexander Rieder\thanks{University of Vienna, Faculty of Mathematics, Oskar-Morgenstern-Platz 1, 1090 Wien, Austria}
}
 
\title{Metastable Speeds in the Fractional Allen-Cahn Equation}

\maketitle

\begin{abstract}
We study numerically the one-dimensional Allen-Cahn equation with the spectral fractional Laplacian $(-\Delta)^{\alpha/2}$ on intervals with homogeneous Neumann boundary conditions. 
In particular, we are interested in the speed of sharp interfaces approaching and annihilating each other. 
This process is known to be exponentially slow in the case of the classical Laplacian. 
Here we investigate how the width and speed of the interfaces change if we vary the exponent~$\alpha$ of the fractional Laplacian. 
For the associated model on the real-line we derive asymptotic formulas for the interface speed and time--to--collision in terms of~$\alpha$ and a scaling parameter~$\varepsilon$.
We use a numerical approach via a finite-element method based upon extending the fractional Laplacian to a cylinder in the upper-half plane, 
and compute the interface speed, time--to--collapse and interface width for~$\alpha\in(0.2,2]$.
A comparison shows that the asymptotic formulas for the interface speed and time--to--collision give a good approximation for large intervals.
\end{abstract}

\section{Introduction}
\label{sec:intro}

In this work we study the fractional Allen-Cahn equation
\be
\label{eq:AC}
\partial_t u = -\eps^\alpha(-\Delta)^{\alpha/2}u+u(1-u^2)=:
-\eps^{\alpha}(-\Delta)^{\alpha/2}u-f(u),
\ee
for $u=u(x,t)$, where $(x,t)\in\Omega\times[0,T)$, $\Omega:=[-L,L]$ is the spatial 
domain for $L>0$, $-(-\Delta)^{\alpha/2}$ is the fractional Laplacian for
$\alpha\in(0,2]$, and $0<\eps\ll 1$ is a small parameter. 
We assume homogeneous Neumann boundary conditions and an initial condition
\benn
\partial_xu(-L,t)=0=\partial_x u(L,t),\qquad u(x,0)=u_0(x) \,,
\eenn
for some given initial datum $u_0 =u_0(x)$.
We consider the \textit{spectral} fractional Laplacian, which is defined as
\benn
(-\Delta)^{\alpha/2} u:=\sum_{n=0}^{\infty}{\lambda_{n}^{\alpha/2} (u,\varphi_n)_{L^2(\Omega)} \, \varphi_n}
\eenn
where $(\lambda_n,\varphi_n)$ are the eigenvalues and eigenfunctions of the Laplacian with homogeneous Neumann boundary conditions
scaled such that $\|\varphi_n\|_{L^2(\Omega)} = 1$. 
Other possible definitions for fractional Laplacians on bounded domains with Neumann-type boundary conditions are discussed and compared in~\cite[\S 7]{DiRoVa:2017} and~\cite[\S 6]{Grubb:2016}.

\smallskip 
The potential $F$ associated to $f(u) =u(u^2 -1)$ via $F'(u)=f(u)$ is given by
\benn
F(u):=\frac14 u^4-\frac12 u^2 \,.
\eenn
There are three homogeneous steady states for~\eqref{eq:AC} given by $u_*=-1,0,1$. 
Using the linearized problem 
\benn
\partial_t U=\underbrace{\left[-\eps^\alpha(-\Delta)^{\alpha/2}
-\txtD_u f(u_*)\right]}_{:=L(u_*)}U,\qquad U=U(x,t),
\eenn
one easily checks that $u_*=\pm1$ are locally asymptotically stable since the spectrum of $L(u_*)$ is contained in $\{z\in\C:\Re(z)<0\}$. 
The state $u_*=0$ is unstable. 
Indeed, one can also view $u_*=\pm 1$ as global minima and $u_*=0$ as a local maximum of the potential $F$.

For the case of $\alpha=2$, so that $-(-\Delta)^{\alpha/2}=\Delta$, even more global dynamics of~\eqref{eq:AC} is well understood; see Section~\ref{subsec:classical} for a more detailed technical review. 
Here we just emphasize that the Allen-Cahn equation~\eqref{eq:AC} can exhibit metastability as shown in Figure~\ref{fig:alpha2}. 
\begin{figure}[h]
  \begin{subfigure}[b]{0.32\textwidth}
    \includegraphics[width=\textwidth]{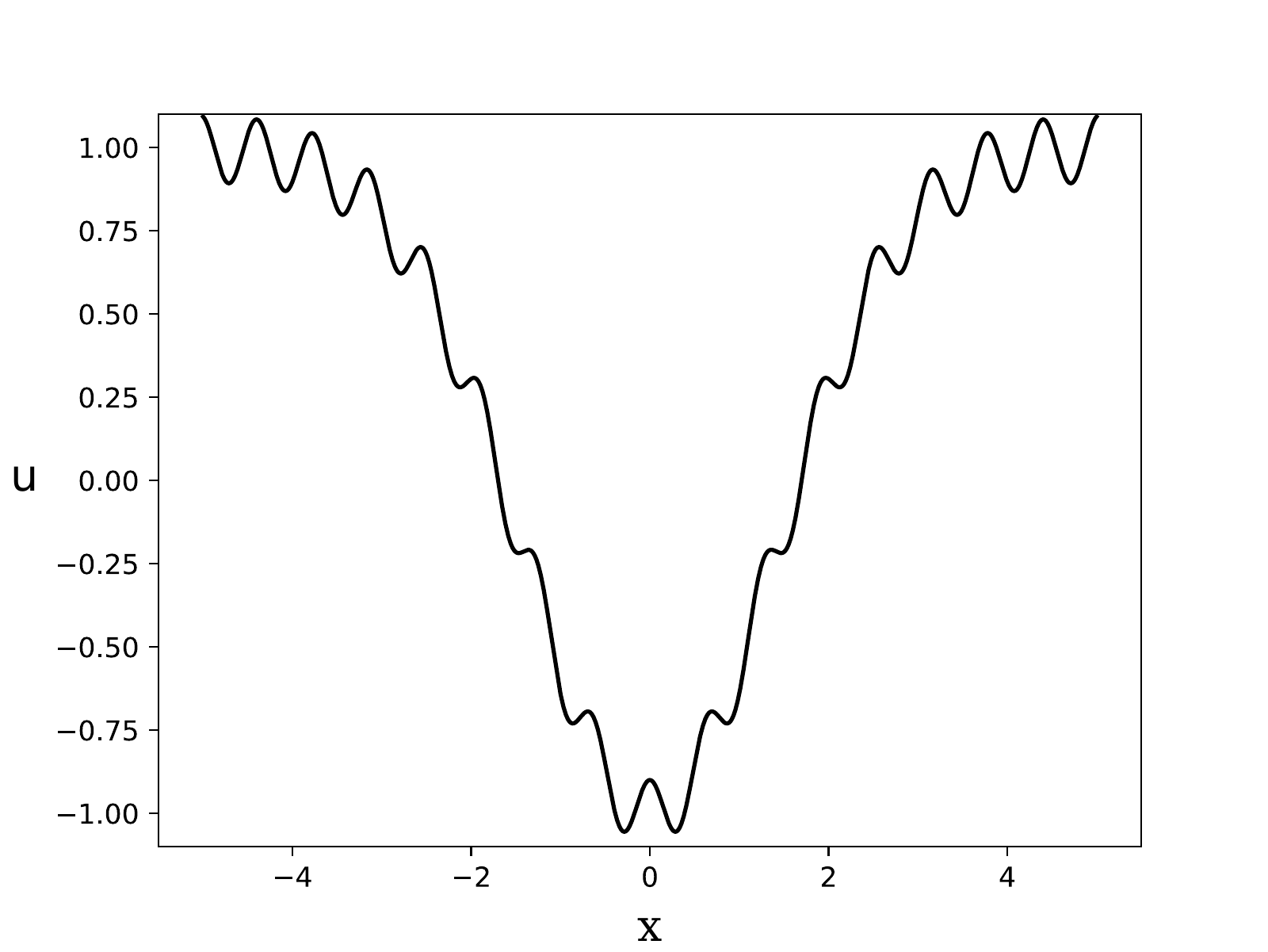}
    \caption{solution at time $t=0$.}
  \end{subfigure}
  \begin{subfigure}[b]{0.32\textwidth}
    \includegraphics[width=\textwidth]{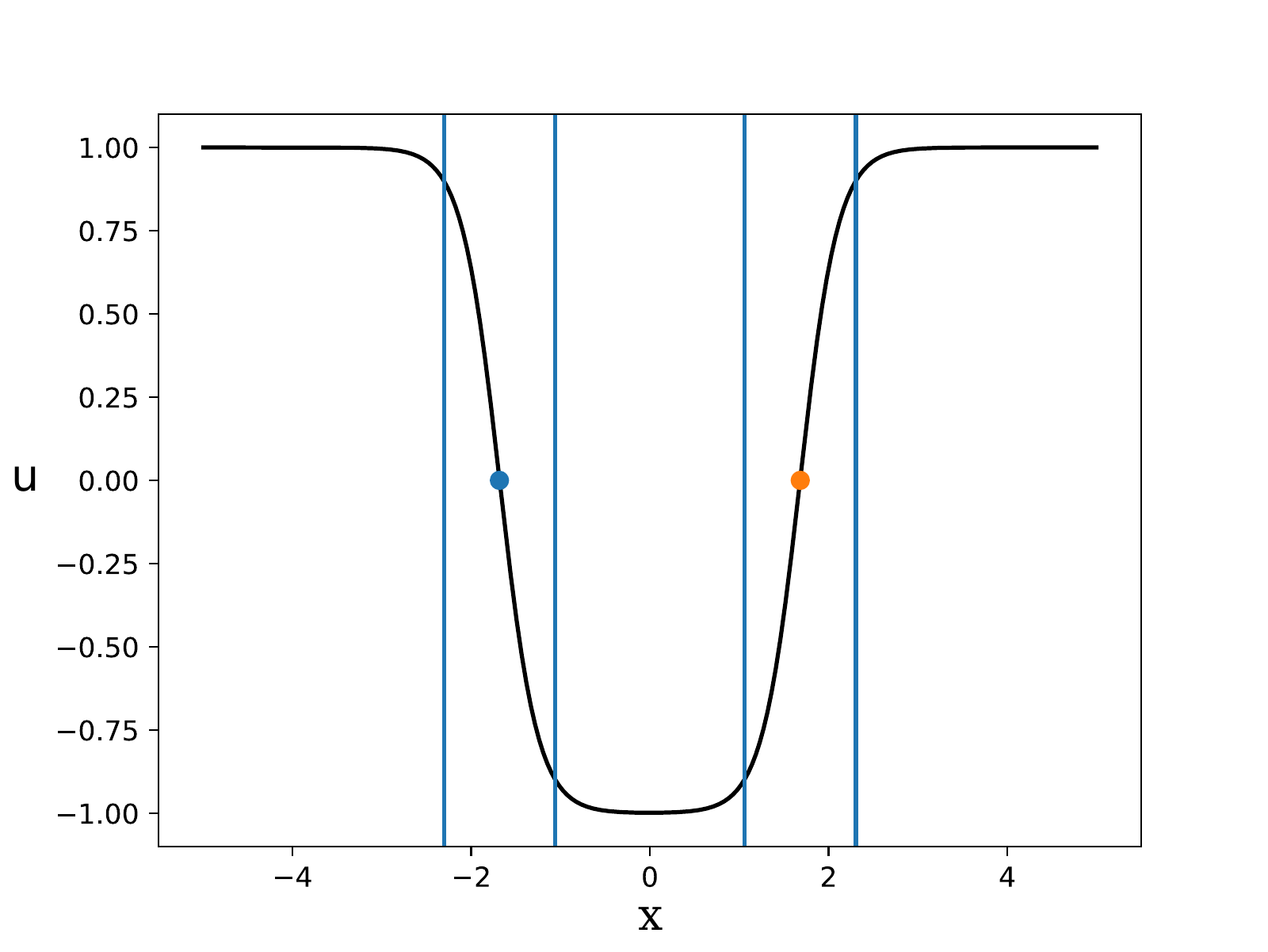} 
    \caption{solution at time $t=1771$.}
  \end{subfigure} 
  \begin{subfigure}[b]{0.32\textwidth}
    \includegraphics[width=\textwidth]{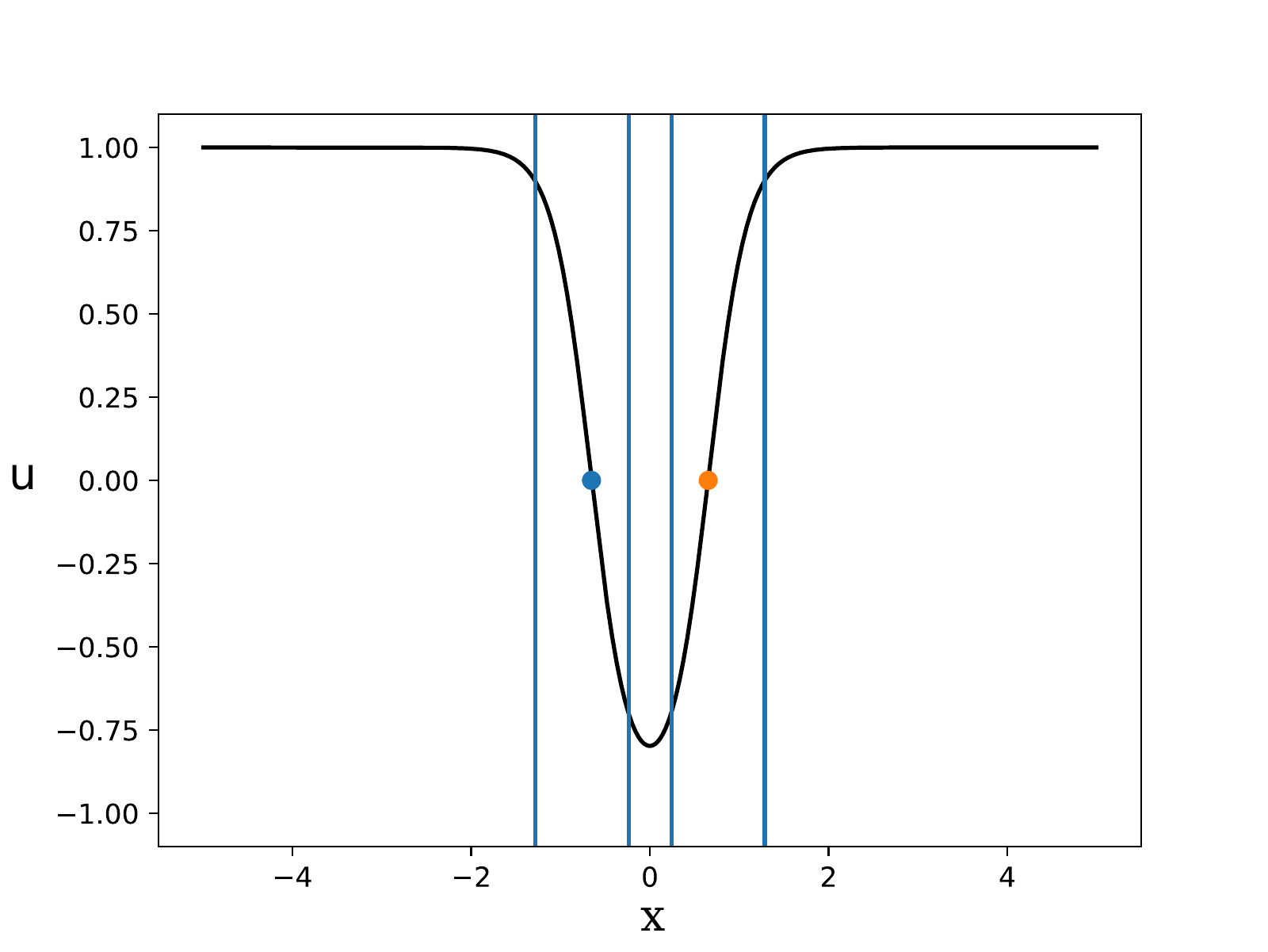}
    \caption{solution at time $t=177216$.}
  \end{subfigure}
  \caption{Evolution of solution for the classical Allen-Cahn equation~\eqref{eq:AC}
    for~\eqref{eq:TL:simple} with $L=5$, $\eps=0.3$ and $\alpha=2$.}
  \label{fig:alpha2}
\end{figure}

We observe that suitable initial conditions $u_0$ get attracted very quickly to solutions composed of multiple, say $n$, very sharp interfaces. 
Each interface is known to have width $\cO(\eps)$ for $\alpha=2$; see Section~\ref{subsec:classical}.
For a long time, this $n$-interface solution appears to be stationary, but it evolves on an exponentially long time scale $\cO(\txte^{K/\eps})$ for some constant $K>0$; see again Section~\ref{subsec:classical} for details. 
On this exponentially long time scale the interfaces move towards each other and then annihilate; see Figure~\ref{fig:alpha2}. 
This effect is the essence of \emph{metastable behavior}, i.e., apparent stationarity on extremely long, yet still transient, time scales.

\medskip 
In this work, we are interested in the influence of the parameter $\alpha$ on metastable interface motion. 
It is natural to ask how the order of the fractional Laplacian influences the interface speed as well as  the interface width. 
For $\alpha=2$, these effects have been quantified, see, e.g.,~\cite{CarrPego}.
However, for $\alpha=(0,2)$ a precise quantification has not been carried out yet.

However, this problem has been studied for~\eqref{eq:AC} with the fractional Laplacian on the real line.
For $\alpha\in(1,2)$ the dependence of the interface speed has been quantified by heuristic arguments and verified numerically via a pseudo-spectral method~\cite{NecNepGol:2008}.
Moreover, equation~\eqref{eq:AC} with periodic potential $F$ and external forcing on the real line has been studied as a model in crystal dislocation dynamics, see~\cite{GoMo:2012, DiFiVa:2014, DiPaVa:2015, PaVa:2015, PaVa:2016, PaVa:2017}.
They characterize the long-time behavior of solutions which includes metastable scenarios and give rigorous proofs.

It is not obvious if these results carry over to the fractional Allen-Cahn equation~\eqref{eq:AC} with homogeneous Neumann boundary conditions.
For example, the fractional Laplacian on the real line is uniquely defined, e.g. see~\cite{Kwa:2017}, whereas on bounded intervals several distinct definitions exist, see~\cite{SerVal:2014, Grubb:2016, DiRoVa:2017, Lischke+etal:2020}.
However our numerical studies of the interface motion for~\eqref{eq:AC} with homogeneous Neumann boundary conditions indicate that the interface speed is governed by the same asymptotic dependence on~$\eps$ and~$\alpha$. 
Additionally, we characterize the dependence of the interface width on these parameters.
In particular, our main observations are:
\begin{itemize}
 \item[(R1)] \textit{Interface speed:} 
  For $\alpha\in(0.5,2)$, the magnitude of the interface speed in case of two interfaces is approximately
  \begin{equation} \label{interface.speed}
   |s (\eps,\alpha)|
   \approx \eps^{1+\alpha}\ \frac{4}{\alpha} \frac{2^{\alpha} \Gamma((1+\alpha)/2)}{\sqrt{\pi} |\Gamma(-\alpha/2)|}\ \gamma\ \frac{1}{|x_1 -x_2|^{\alpha}} \,.
  \end{equation}
  where $x_1$ and $x_2$ are the centers of the interfaces and $\gamma := (\int_\R (v'(x))^2 \txtd x)^{-1}$ is the inverse of the semi-norm of a basic layer solution~$v$ of~\eqref{eq:BLS:1}--\eqref{eq:BLS:2}.
 \item[(R2)] \textit{Interface width:} 
 For $\alpha\in(0.2,2)$, the interface width is approximately
 \begin{equation} \label{interface.width}
  w(\eps,\alpha) = b(\alpha) \eps^{\kappa_1 \alpha^{-1} +  \kappa_2}\,,
 \end{equation}
where we numerically estimated $\kappa_1\approx -0.168298$ and $\kappa_2\approx 1.11709$. We note that the numerically estimates suggest the \emph{conjectures} $\kappa_1=1/6$ and $\kappa_2=10/9$. 
\end{itemize}

The paper is structured as follows: 
In Section~\ref{sec:analysis}, we briefly review analytical results for the Allen-Cahn equation involving the Laplacian ($\alpha=2$) with homogeneous Neumann boundary conditions and the fractional Laplacian ($\alpha\in(0,2)$) on the real line, respectively. 
In Section~\ref{sec:numerics}, we describe the numerical setup used to simulate~\eqref{eq:AC} for $\alpha\in(0,2]$. 
Moreover, we compare our numerical method with a spectral method in~$\S$\ref{subsec:comparision_methods}.
The main results are presented and discussed in Section~\ref{sec:results}. 
A brief summary and an outlook to future open problems can be found in Section~\ref{sec:outlook}.
 
\section{Asymptotic analysis}
\label{sec:analysis}

\subsection{The Classical Case}
\label{subsec:classical}

In this section, we briefly review the known results for metastability of the
classical Allen-Cahn equation
\be
\label{eq:AC1}
\partial_t u = \eps^2\partial_{xx} u+u(1-u^2),
\ee
with Neumann boundary conditions posed on the interval $[0,1]$ as studied 
by several authors; here we mainly follow~\cite{CarrPego}. 
Note that considering the interval $[0,1]$ is equivalent to picking $[-L,L]$ up to shifting and 
scaling the $x$-coordinate. 
A key auxiliary tool to construct $n$-interface solutions is to first construct a single layer. 
This requires the solution of the auxiliary two-point boundary value problem
\be
\label{eq:BVP}
\eps^2 \frac{\txtd^2 \phi}{\txtd x^2}=\phi(1-\phi^2),\qquad 
\phi(-\ell/2)=0=\phi(\ell/2).
\ee   
For $\eps>0$ sufficiently small, it can be proven that~\eqref{eq:BVP} has 
a unique solution $\phi(x,\ell)$ for each given $\ell>0$, which is positive for 
$|x|<\ell/2$. 
The shape of $\phi(x,\ell)$ corresponds to a spike centered at $x=0$, i.e., an interior layer solution in the terminology of multiscale analysis~\cite{KuehnBook}. 
The width of this spike is $\cO(\eps)$ and it can also be interpreted as a homoclinic orbit in $(\phi,\phi')$-coordinates to $(0,0)$. 
Next, one wants to construct an interface (layer) with location $h=h(t)$ for~\eqref{eq:AC1}. 
Consider a monotone function 
\benn
\xi\in C^\I(\R,[0,1]),\qquad 
\xi(x)=\left\{
\begin{array}{ll}
0\quad\text{for $x\leq -1$,}\\
1\quad\text{for $x\geq 1$,}\\
\end{array}
\right.
\eenn
and define the \emph{approximate metastable interface}, 
$u^h=u^h(x)$ for some $0<h<1$ by
\be
\label{eq:hlayer}
u^h(x):=-\left[1+\xi\left(\frac{x-h}{\eps}\right)\right]
\phi(x,2h)+\xi\left(\frac{x-h}{\eps}\right)\phi(x-1,2-2h)
\ee
so that $u^h(0)<0$, $u^h(h)=0$, $u^h(1)>0$, and the width of the interface is $\cO(\eps)$.
Obviously one can construct an approximate $n$-interface solution by a similar procedure. 
Since any $n$-interface solution has $n$ natural coordinates given by the positions $h$ of each interface,
one can define an $n$-dimensional manifold $\cM$ in $H^1([0,1])$ parametrized by the positions. 
It is shown in~\cite{CarrPego} that this manifold is locally attracting for large classes of initial data for~\eqref{eq:AC1}. 
Once the solution is close to the approximately invariant manifold $\cM$, there exists an ODE describing the motion of the positions $h$ on $\cM$. 
It turns out that all interface positions move at equal asymptotic speed in a generic setting away from any annihilation events. 
Their speed in this regime can be determined by just looking at one interface. 
It is highly non-trivial to prove that the ODE for its position is given by 
\be
\label{eq:hmotion}
h'=K_0\eps \left[q\left(\frac{\eps}{2h}\right)
-q\left(\frac{\eps}{2-2h}\right)\right]+R_1,\qquad q(r):=F(\phi(0,\ell)),~
r:=\eps/\ell,
\ee  
where $K_0>0$ is a computable order $\cO(1)$ constant, and $R_1$ turns out to be a 
negligible remainder term~\cite{CarrPego}. 
Metastability arises because one can also prove that
\benn
q\left(r\right)=K_1\txte^{-K_2 \ell/\eps}+R_2,
\eenn
for constants $K_{1,2}>0$ and $R_2$ turns out to be another higher-order terms irrelevant for computing the leading-order of the vector field in~\eqref{eq:hmotion}, i.e., the local interface speed is exponentially small for $\alpha=2$. 

\subsection{Analysis of fractional Allen-Cahn equations}
For $\alpha\in(0,2)$, the results for the classical case lead to the following questions:
\begin{itemize}
 \item[(Q1)] Does the metastable interface speed $s=s(\eps,\alpha)$ change?
 If so, what is the graph of $s(\eps,\alpha)$ as a function of $\alpha$ and/or $\eps$? 
 \item[(Q2)] Does the interface width $w=w(\eps,\alpha)$ change? 
 If so, what is the graph of $w(\eps,\alpha)$ as a function of $\alpha$ and/or $\eps$? 
\end{itemize}
The questions (Q1)-(Q2) are our main focus as they provide immediate information
on the relevant dynamical behavior. 

We discuss the ambitious task to extend the results by Carr and Pego in the Section~\ref{sec:outlook}.
Next, we will present the analysis of a fractional Allen-Cahn equation on $\R$,
which provides approximations for the interface speed and time--to--collapse.
In Section~\ref{sec:results} we show by numerical simulations that these formulas also give a good approximation for our model on a bounded interval.

\subsubsection{Analysis of fractional Allen-Cahn equations on $\R$}
Equation~\eqref{eq:AC} with periodic potential $F$ and external forcing on the real line $x\in\R$ models the dynamics of crystal dislocations, see~\cite{GoMo:2012, DiFiVa:2014, DiPaVa:2015, PaVa:2015, PaVa:2016, PaVa:2017}.
First, the evolution of dislocations given as a superposition of transitions with the same orientation has been studied. It has been shown that these transitions repel each other~\cite{GoMo:2012, DiFiVa:2014, DiPaVa:2015}.
Patrizi and Valdinoci~\cite{PaVa:2015, PaVa:2016, PaVa:2017} considered also dislocations given as a superposition of transitions with arbitrary orientations and studied again the long-time behavior.
For example, two transitions of opposite orientation attract each other (if no external force is present), which shows the metastable behavior we are interested in.
In general, they study the long-time behavior of well--prepared initial data modeling an arbitrary (but finite) number of transitions with any order of orientations.
We report here their results for the equation
\be
\label{eq:AC:rescaled}
\partial_t u_\eps = \tfrac1{\eps} \Big(-(-\Delta)^{\alpha/2}u_\eps -\tfrac1{\eps^\alpha} F'(u_\eps)\Big),
\ee
on $x\in\R$ with (non-periodic) potential $F(u) =\tfrac{u^4}4 -\tfrac{u^2}2$ and superposition of transitions with alternating orientations (and note the minor modifications needed along the way):

\medskip
Equations~\eqref{eq:AC} and~\eqref{eq:AC:rescaled} are related via a rescaling of time $t\mapsto t/\eps^{1+\alpha}$.

\medskip
First, a well--prepared initial datum $u_0$ is constructed from basic layer solutions of~\eqref{eq:AC} (instead of initial layer solutions):
A basic layer solution~$v:\R\to[-1,1]$ is a stationary solution of~\eqref{eq:AC}, i.e. it solves 
\begin{equation} \label{eq:BLS:1} 
 -(-\Delta)^{\alpha/2}v -f(v) = 0 \,, \qquad x\in\R\,, 
\end{equation}
and satisfies
\begin{equation} \label{eq:BLS:2} 
 v'(x)>0 \text{ for all } x\in\R\,, \quad v(-\infty)=-1\,, \quad v(0)=0\,, \quad v(+\infty)=1\,.
\end{equation}
Here, with a slight abuse of notation $-(-\Delta)^{\alpha/2}$ denotes the fractional Laplacian acting on functions on the real line.
Note that the fractional Laplacian given on the real line is equivalent to a singular integral representation, see~\cite{CafSil:2007,Kwa:2017}.
The existence and uniqueness (due to $v(0)=0$) of a basic layer solution has been proved in~\cite[Thm. 2.4]{CabSir:2015}.

\medskip
Then, a well--prepared initial datum $u_0$ is constructed from shifted basic layer solutions~$v$ of alternating orientation.
For example, we study the evolution of a transition layer given by 
\begin{equation} \label{eq:TL:simple} 
 u_\eps^0(x)
 := v\bigg(\frac{x-x_1^0}{\eps}\bigg) + v\bigg(-\frac{x-x_2^0}{\eps}\bigg) +1
\end{equation}
for some $x_2^0 < x_1^0$.
More generally, we study the evolution of an initial datum $u_\eps^0$ given by
\begin{equation} \label{eq:TL} 
 u_\eps^0(x) := \sum_{i=1}^{2K} v_{\eps,i}^0(x) +1
\end{equation}
where $K\in\N$, $v_{\eps,i}^0(x) = v\Big(\zeta_i \frac{x-x_i^0}{\eps}\Big)$ with $x_{2K}^0 < x_{2K-1}^0 < \ldots < x_2^0 < x_1^0$ and
\[ \zeta_i = \begin{cases} +1 &\text{if $i$ is odd,} \\ -1 &\text{if $i$ is even,} \end{cases} 
     \qquad i\in\{1,2,\ldots,2K\} \,.
\]
 
It can be shown that, for sufficiently small $\eps>0$, the solution $u_\eps(x,t)$ of~\eqref{eq:AC:rescaled} with initial datum $u_\eps^0$ is approximately of the form
\begin{equation}
  u_\eps(x,t) 
   \approx \sum_{i=1}^{2K} v\Big(\zeta_i \frac{x-x_i(t)}{\eps}\Big) +1 \,,
\end{equation}
where the functions $x_i(t)$ satisfy approximately (i.e., up to higher order terms in $\eps$) a system of ODEs
\begin{equation} \label{ODE:centers}
 \begin{cases}
   \frac{\txtd x_i}{\txtd t} = \frac{\eps^{1+\alpha}}{\alpha} 4 \frac{2^{\alpha} \Gamma((1+\alpha)/2)}{\sqrt{\pi} |\Gamma(-\alpha/2)|} \gamma \sum_{i\ne j} \zeta_i \zeta_j \frac{x_i -x_j}{|x_i -x_j|^{1+\alpha}} \,, \\
   x_i(0) = x_i^0 \,,
 \end{cases}
\end{equation}
for $i\in\{1,2,\ldots,2K\}$ and $\gamma := (\int_\R (v'(x))^2 \txtd x)^{-1}$.
The centers $x_i$ are initially ordered as $x_{2K}^0 < x_{2K-1}^0 < \ldots < x_2^0 < x_1^0$, then for sufficiently small times $t>0$ this order will persist for the solution of~\eqref{ODE:centers}, i.e.
\begin{equation} \label{ineq:centers}
  x_{2K}(t) < x_{2K-1}(t) < \ldots < x_2(t) < x_1(t) \,.
\end{equation}
System~\eqref{ODE:centers} is well-defined until the first collision at time $T_C$, which is defined as the time such that~\eqref{ineq:centers} holds for all $t\in[0,T_C)$ and
\begin{equation} \label{cond:TC}
 \exists i_C \in\{1,2,\ldots,2K-1\}\,: \qquad
 x_{i_{C} +1}(T_C) =x_{i_C}(T_C)\,. 
\end{equation} 

More precisely, it is proved that the solution $u_\eps$ approaches a superposition of sharp transitions with moving centers $x_i(t)$ in the following sense:
\begin{thm}[cf. {\cite[Thm. 1.1]{PaVa:2015}}] \label{thm:ueps}
Suppose $\alpha\in[1,2]$.
Let 
\begin{equation}\label{eq:AL}
 v(t,x) = \sum_{i=1}^{2K} \sgn(\zeta_i(x-x_i(t))) +1\,,
\end{equation}
where $\sgn$ is the sign function and $(x_i(t))_{i=1,\ldots,2K}$ is the solution to~\eqref{ODE:centers}.
Then, for every $\eps>0$ there exists a unique solution $u_\eps$ of~\eqref{eq:AC:rescaled}.
Furthermore, as $\eps\to0^+$, the solution $u_\eps$ exhibits the following asymptotic behavior:
\begin{align}
 \label{est:above}
 \limsup_{\stackrel{(t',x')\to(t,x)}{\eps\to 0^+}} u_\eps(t',x') 
 \leq \limsup_{(t',x')\to(t,x)} v(t',x')
\intertext{and}
 \label{est:below}
 \liminf_{\stackrel{(t',x')\to(t,x)}{\eps\to 0^+}} u_\eps(t',x') 
 \geq \liminf_{(t',x')\to(t,x)} v(t',x')
\end{align}
for any $(t,x)\in[0,T_C)\times\R$.
\end{thm}
\begin{proof}[Sketch of proof]
 The initial conditions $u_\eps^0$ under consideration take values in $[-1,1]$.
 Due to a maximum principle for~\eqref{eq:AC:rescaled}, the corresponding solution $u_\eps$  of~\eqref{eq:AC:rescaled} with initial datum $u_\eps^0$ will again satisfy $-1\leq u_\eps (t,x)\leq 1$ for all $(t,x)$.
 We consider~\eqref{eq:AC:rescaled} with the periodic potential $F_{per}$ such that $F_{per}(u) =F(u)$ for all $u\in[-1,1]$. 
 Then $F_{per}\in C^{2,\beta}(\R)$ is H\"older continuous.
 Now we can use the result of~\cite[Thm. 1.1]{PaVa:2015} for the modified problem.
 However, since $-1\leq u_\eps (t,x)\leq 1$ for all $(t,x)$, the statement carries over to the original problem unchanged.
\end{proof}
\begin{remark} 
\begin{enumerate}
\item The restriction to $\alpha\in[1,2]$ seems to be due to technical conditions in the proofs. In case of $\alpha\in(0,1)$, the assumption $F\in C^{3,\beta}(\R)$ for some $\beta>0$ is needed, see~\cite{DiFiVa:2014}.
\item In the studies on crystal dislocation dynamics the stable stationary states are taken to be 0 and 1 compared to -1 and 1 as in our setting.  
That means, we have to include e.g. an affine transformation $y \mapsto -2y+1$ which gives an additional multiplicative constant~$4$.
\item
Moreover, we have to include a normalization factor $\frac{2^{\alpha} \Gamma((1+\alpha)/2)}{\sqrt{\pi} |\Gamma(-\alpha/2)|}$, due to the singular integral representations used in~\cite[Lemma 5.1]{stinga_torrea} compared with~\cite[(1.2)]{PaVa:2015}.  
\end{enumerate}
\end{remark}

In case of two initial transitions with opposite orientation the following estimate for the time--to--collision~$T_C$ holds 
\begin{thm}[{cf. \cite[Thm. 1.2]{PaVa:2015}}] \label{thm:TC}
 Let $K=1$. Let $d_0 := x_1^0 -x_2^0 >0$.
 Then, 
 \begin{equation} \label{eq:TC}
 T_C = \tfrac{C_\alpha d_0^{1 +\alpha}}{2(1 +\alpha)\gamma}
 \qquad\text{with } 
  C_\alpha 
  :=1/\bigg( \frac{4}{\alpha} \frac{2^{\alpha} \Gamma((1+\alpha)/2)}{\sqrt{\pi} |\Gamma(-\alpha/2)|}\bigg)
  =\frac{\Gamma(1-\alpha)}{2\sec(\alpha \frac{\pi}{2})}\,.
 \end{equation}
\end{thm}

Then,~\eqref{est:above} and~\eqref{est:below} imply that for any $x\ne x_c$ where $x_c:=x_1(T_C)=x_2(T_C)$, we have 
\[ \lim_{t\to T_C^-} \lim_{\eps\to 0^+} u_\eps(t,x) =1 \,. \]
That means two dislocations annihilate each other after their collision.
However, the limit of $u_\eps(t,x)$ keeps a memory of them, in the sense that $u_\eps$ at the point $x_c$ does not approach the state~$1$ in the limit.
\begin{thm}[{\cite[Thm. 1.4]{PaVa:2015}}] \label{thm:xc}
 Let $K=1$. Let $u_\eps$ be the solution to~\eqref{eq:AC:rescaled}, then
 \begin{equation}\label{est:limsup:below}
 \liminf_{\stackrel{t\to T_C^{-}}{\eps\to 0^+}} u_\eps(t,x_c) \leq -1 \,.
 \end{equation}
\end{thm}
Nonetheless, after some time $\Teps$ -- which is only slightly larger than the collision time $T_C$ -- the solution $\ueps$ will become small like $\eps$. More precisely,
\begin{thm}[cf. {\cite[Thm. 1.1 \& Thm. 1.2]{PaVa:2016}}] \label{thm:Teps}
Let $K=1$. 
Under the assumptions of Theorem~\ref{thm:ueps}, let $\ueps$ be a solution of~\eqref{eq:AC:rescaled}.
Then, there exists $\eps_0>0$ such that for all $\eps<\eps_0$ there exist $\Teps,\rho_\eps>0$ such that
\[ \Teps = T_C +o(1)\,, \quad \rho_\eps =o(1) \quad \text{as $\eps\to 0$} \]
and
\begin{equation}
 1-\ueps(\Teps,x) \leq \rho_\eps \quad \text{for all $x\in\R$.}
\end{equation}
Moreover, there exist $\tilde{\eps}_0>0$ and $c>0$ such that for any $\eps<\tilde{\eps}_0$ we have 
\begin{equation}
 |1-\ueps(t,x)| \leq \rho_\eps \exp\big( c\tfrac{\Teps -t}{\eps^{1+\alpha}}\big)\,, \quad \text{for all $x\in\R$ and $t\geq\Teps$.}
\end{equation}
\end{thm}

\section{Numerical Methods}
\label{sec:numerics}

In recent years the numerical simulations of partial differential equations with fractional space derivatives has attracted a lot of attention. 
Some references on (various forms of) fractional Allen-Cahn equations in one and multi-dimensional spatial settings are the following:~\cite{BurHalKay:2012,Zeng+etal:2014,StoYue:2015,AchleitnerKuehn2,SimYanMor:2015,SonXuKar:2016,ZhaWenFen:2016,HouTanYan:2017,AlzKha:2019,LeeLee:2019,LiuCheWan:2020,HePanHu:2020}.

In order to discretize \eqref{eq:AC}, we consider an alternative formulation of the fractional Laplacian, based on the Caffarelli-Silvestre extension~\cite{CafSil:2007}, as presented in \cite{stinga_torrea} for the Neumann problem.
Consider the solution $\UU: \R_+ \times \Omega  \times \R_+ \to \R$ to:
\begin{subequations}
  \label{eq:lifted_eq}
\begin{align}
  -\operatorname{div}\left( y^{1-\alpha} \UU(t) \right) &=0\,,  && \text{ in $\Omega \times \R_+$ },\ \forall t>0\,, \\
  d_\alpha \trace \UU_t + \eps^2 \partial_\nu^\alpha \UU   &=-d_\alpha f(\trace \UU )\,, && \text{ in $\Omega \times \{0\} \times \R_+$ }, \\
  \partial_\nu \UU(t) &=0\,,  &&\text{ on $\partial \Omega \times \R_+$},\ \forall t>0\,,\\
  \trace \, \UU(0)&=u_0\,, && \text{ in $\Omega$.}                   
\end{align}
\end{subequations}
Here $d_\alpha:=2^{1-\alpha}\Gamma(1-\alpha/2)/\Gamma(\alpha/2)$, $\partial_\nu$ denotes the
normal derivative with respect to the lateral boundary,
$\trace $ denotes the trace operator with respect to the artificial variable $y$ at $0$  and $\partial_\nu^\alpha$ is defined by
\begin{align*}
  \partial_\nu^{\alpha}  \UU:=  -\lim_{y\to 0^+}{y^{1-\alpha} \partial_y \, \UU(\cdot,y)}.
\end{align*}
It can then be proven that if $\UU$ solves (\ref{eq:lifted_eq}), then $\trace \UU$ solves~\eqref{eq:AC}.

The discretization is based on an semi-implicit BDF2 discretization of the time variable, and a $hp$-finite element method to discretize the $x$ and $y$ variables, similar to what was presented in~\cite{tensor_fem} for the linear stationary case and~\cite{pde_frac_parabolic} for the linear time-dependent fractional diffusion problem, but which is based on a low order
discretization in space. An analogous $hp$-type method for the linear fractional heat equation,
but using different time-stepping and Dirichlet boundary condition was analyzed in detail in
\cite{fractional_heat}.

The semi-implicit BDF2 (SBDF2) method for the ODE  $u_t = \mathscr{L} u + f(t,u)$ with stiff operator $\mathscr{L}$ is given by
\begin{align*}
u^{n+1}&=\frac{4}{3} u^n - \frac{1}{3} u^{n-1} + \frac{2}{3} \Delta t\left[2 f(t_n,u^n) - f(t_{n-1},u^{n-1})\right] + \frac{2}{3} \Delta t \mathscr{L} u^{n+1}.
\end{align*}
For the discretization in the auxiliary variable $y$, we fix $\sigma \in (0,1)$ and 
consider a graded mesh $\cT_y:=\{y_0,\dots y_N\}$ with $y_0:=0$, $y_{j}:=\cY \sigma^{N-j}$. Here $\cY>0$ is
a cutoff parameter. In $\Omega:=(-L,L)$, we consider a uniform mesh with size $h>0$, denoted by $\cT_x$.
Denoting by $\cS^{p,1}(\cT)$ the space of continuous piecewise polynomials of degree $p$, on  a triangulation $\cT$ we solve in each time step
for $\UU_h^{n+1} \in \cS(\cT_x) \otimes \cS(\cT_y)$ satisfying
\begin{multline*}
  \frac{2\Delta t}{3d_\alpha}\int_{0}^{\cY}{y^{1-\alpha}\left(\nabla_{xy} \UU_h^{n+1}(t),\nabla_{xy} \VV_h\right)_{L^2(\Omega)}\,dy}
  +\left(\trace \UU_h^{n+1},\trace {\VV_h}\right)_{L^2(\Omega)} \\
  = \Big(\frac{4}{3}\trace \UU_h^{n} - \frac{1}{3}\trace \UU_h^{n-1},
    \trace {\VV_h}\Big)_{L^2(\Omega)}    
    +\frac{2\Delta t}{3} \left(f(\trace{\UU}^{n}) - f(\trace{\UU}^{n-1}),\trace \VV_h\right)_{L^2(\Omega)}.
\end{multline*}
for all test functions $\VV_h \in \cS(\cT_x) \otimes \cS(\cT_y)$.
Note that the nonlinearity is treated explicitly.  
Since we expect this contribution to be non-stiff this should not cause concern.
The fact that we replaced the infinite cylinder $\Omega \times \R_+$ with $\Omega \times (0,\cY)$ can be justified by the fact that
when expanding $\UU$ in an eigenbasis of the Laplacian, all contributions (except for the constant one) decay exponentially.
(See~\cite{tensor_fem,fractional_heat} for details in the Dirichlet case, the Neumann case behaves analogously). 
The cutoff of the constant contribution does not impact the result since we used homogeneous Neumann conditions on the artificial boundary, which capture the constant mode exactly.
When referencing this algorithm, we will refer to it as the ``$hp$-method''.

\smallskip
In order to be able to compare the results obtained by our $hp$-FEM type approach, we also implemented and ran our experiments using a simpler spectral type method
proposed in \cite{spectral_code}. 
It directly exploits the spectral definition of the fractional Laplacian by using a discrete cosine transform. 
For time discretization we use an IMEX-type Euler method. 
When referencing this algorithm, we will refer to it as the ``spectral--method''.

\subsection{Detecting an interface and measuring speeds}
In order to extract information on the behavior of the interfaces, we need to fix our methodology for
extracting the quantities of interest from the numerical solutions.
For simplicity, even when using a high order method for solving~\eqref{eq:AC} numerically, our measurements are always taken from a 
piecewise linear approximation to the numerical solution. 
This is done by interpolation on a finer grid.

Since the interfaces need a certain time to form, and we expect the interfaces to collapse after some time, with the scales of these timings depending on $\eps$ and $\alpha$, we use an (empirically determined) estimate for the time  until the collapse happens $t_{col}$, i.e., the solution becomes equal to the constant $1$ function. 
When measuring we then assume that for $\frac{1}{4}t_{col}$ the interface has already formed and assume that afterwards, the interfaces remained stable for a duration of  $\frac{1}{10}t_{col}$ and we can do accurate measurements.

When computing the speed of the interface, we consider the speed of the zeros of the numerical solution.
In order to determine the width of the interface, we made the very simplistic assumption that an interface happens whenever the solution dips between $\max(u)-\delta$ and $\min(u)+\delta$.
For the images chosen here we used $\delta:=0.1$. 
This gave good correspondence of the wave speed determined by the zeros of $u$ and the speed of the midpoint of the interface.
We also tested using $\delta=0.01$ and $\delta \approx \varepsilon^{\alpha}$, but $\delta=0.1$ seemed to give the most robust results.
Both width and speed are computed by taking the average over 100 measurements in the time-frame $\big(\frac{1}{4}t_{col},\frac{7}{20}t_{col}\big)$. Figure~\ref{fig:ex_interfaces} shows a typical situation.
 
\begin{figure}[h]
  \center
  \includegraphics[width=12cm]{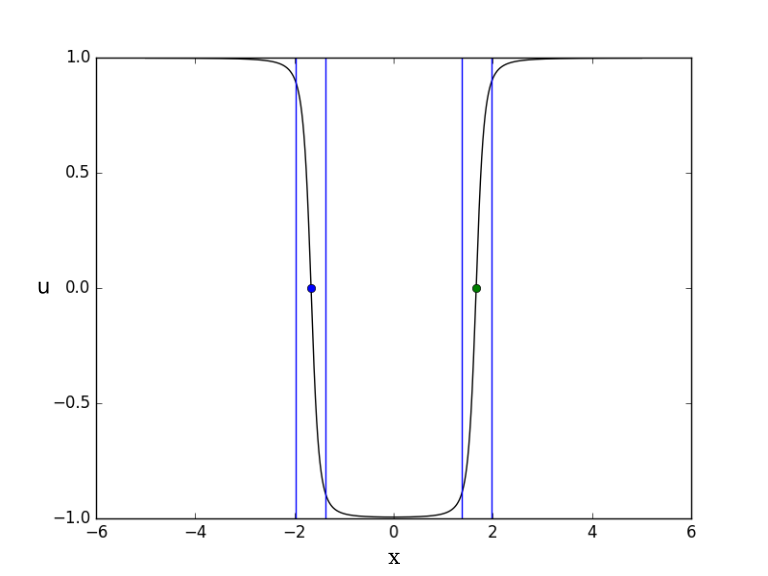}
  \caption{Example for the detected interfaces and roots. }
\label{fig:ex_interfaces}
\end{figure}

In order to determine rates, we used MATLABs \textit{Curve Fitting Toolbox} for polynomial and power-law fits.

\subsection{Model problem}
For our computations we fixed the following parameters.
We considered the domain $\Omega:=(-L,L)$ with $L:=10$.
In order to derive a starting condition which is close to the theoretical
considerations of Section~\ref{sec:analysis}, we approximately computed
a layer solution $v$, as described in Section~\ref{subsec:layer_solution_on_R}.
Tacitly extending the function $v$ by a constant outside of $(-L,L)$,
the initial condition is then given by
\begin{align*}
  u_0(x):=v\Big(\frac{x-x_0}{\eps}\Big) + v\Big(\frac{-x-x_0}{\eps}\Big) +1.
\end{align*}
This function possesses two interfaces at $\pm x_0$, which was taken to be $x_0:=L/4$.
We will later on also discuss what happens for more general initial conditions.
We used a timestepping size of $\Delta t:=\min(0.05,10^{-4}t_{col})$, ensuring that
we make at least $10^4$ steps.

For the spectral method we used a grid with $N=262144$ points. 
For the $hp$-method, in $\Omega$, we used a uniform grid of mesh size $\min(2L/\eps,8012)$ and polynomial degree $p_x=5$ in order to resolve the appearing steep flanks of the interfaces well. To discretize the artificial
variable, we used $p_y=8$ on a geometric grid with $8$ layers and a mesh grading factor of $\sigma=0.125$,
i.e., the grid consists of points of the form $8 (0.125)^{\ell}$ for $\ell=0, \dots 8$.

\subsection{Comparison of the two methods}
\label{subsec:comparision_methods}
In this section, we compare the different methods we used. 
Since they are quite different (first order vs second order in time, different approaches to the fractional Laplacian), we expect that this 
comparison gives a good indication of the real accuracy of our simulations. 
Instead of comparing the discrete solutions directly, we compare the postprocessed quantities: interface speed, interface width and time--to--collapse. 
The results are collected in Figure~\ref{fig:comparison_of_methods}.
As we can see, the error is reasonably small for all contributions, giving confidence that the behavior observed numerically matches the one of the exact solution. The only exception is for small values of $\alpha \approx 0.2$.  As we do not expect our discretization to be robust as $\alpha\to 0$ this has to be expected. At the same time, this behavior may explain the observed mismatch between our predicted behavior and the observations in this regime.

\begin{figure}[h]
  \begin{subfigure}[b]{0.5\textwidth}
    \includeTikzOrEps{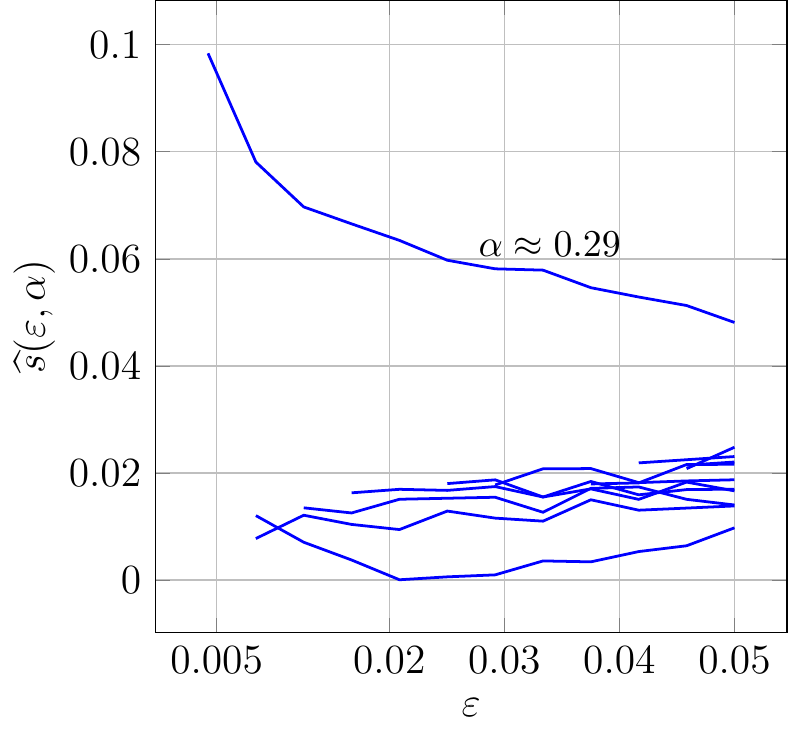}
     \caption{Relative difference of the renormalized interface speed between the methods.}
  \end{subfigure}\qquad 
  \begin{subfigure}[b]{0.5\textwidth}
    \includeTikzOrEps{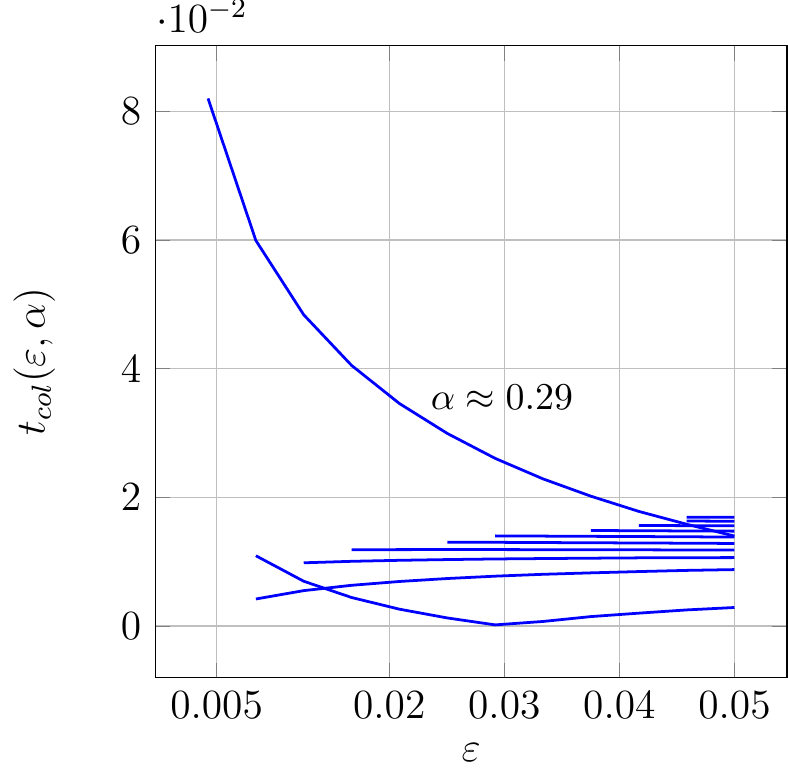} 
    \caption{Relative difference of the time--to--collapse between the methods.}
  \end{subfigure} \\\vspace{2mm}
  \begin{subfigure}[b]{0.4\textwidth}
    \includeTikzOrEps{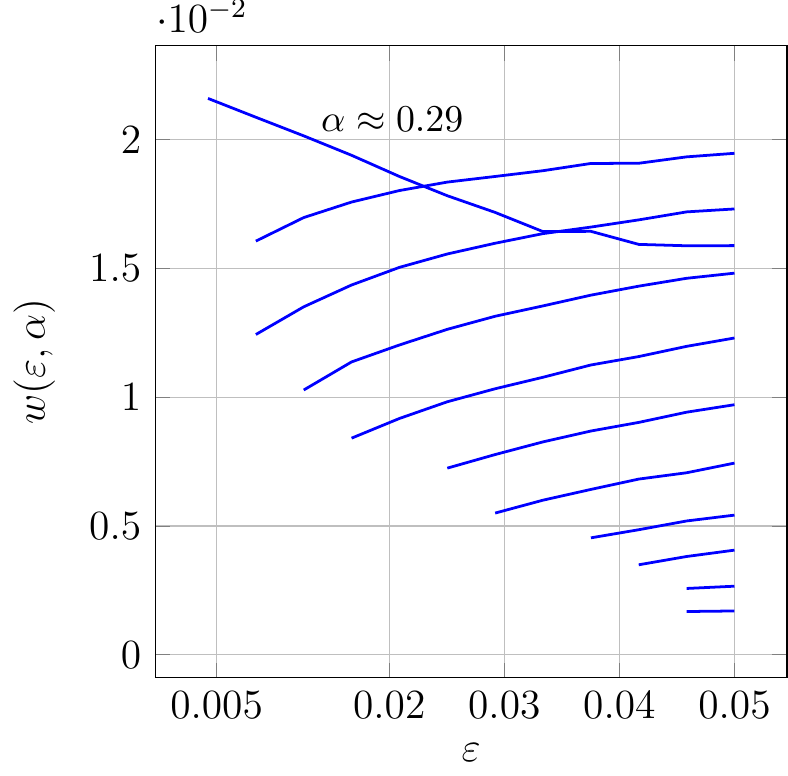}
    \caption{Relative difference of the interface width between the methods.}
  \end{subfigure}
  \caption{Comparison of two numerical methods: \textit{hp}--method vs. spectral--method.
    Different lines represent different values of $\alpha \in (0.2,1.9)$. }
  \label{fig:comparison_of_methods}
\end{figure}

\section{Results on Metastability}
\label{sec:results}
In this section, we consider the evolution of two interfaces.
We present the results of our numerical experiments on intervals $[-L,L]$ and compare them with our analysis on the real line. 
Therefore, it is natural to investigate the dependence of our quantities of interest on the size of the considered interval $[-L,L]$. 
As can be seen in Figure~\ref{fig:dependence_on_domain}, neither the (modified) speed nor the width of the interface seem to depend strongly on the size of the domain. 
Especially for larger domains, the results become indistinguishable.

\begin{figure}
  \centering
  \begin{subfigure}[b]{0.45\textwidth}
    \includeTikzOrEps{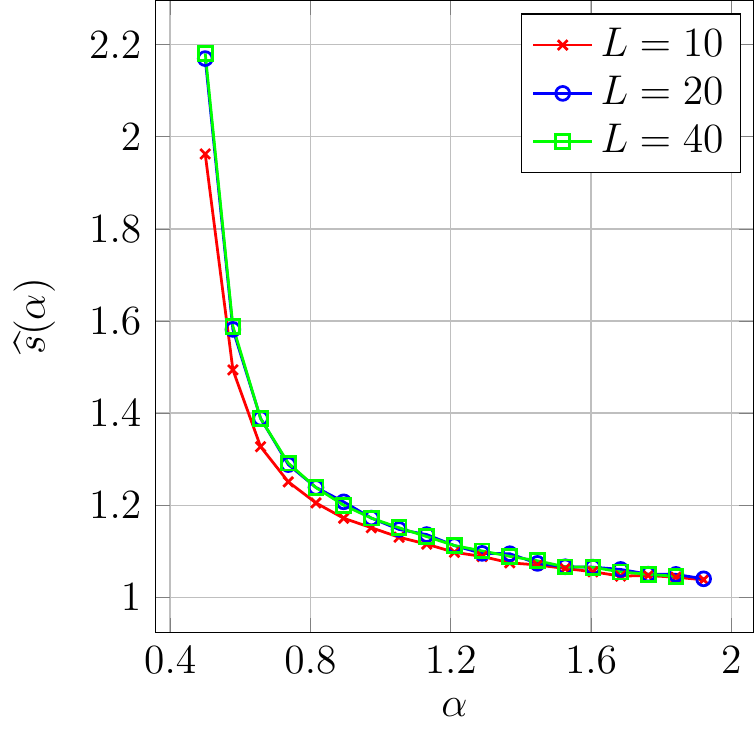}
    \caption{Modified interface speed $\widehat{s}$.}
    \label{fig:comp_speed_on_intervals}
  \end{subfigure}
  \begin{subfigure}[b]{0.45\textwidth}
    \includeTikzOrEps{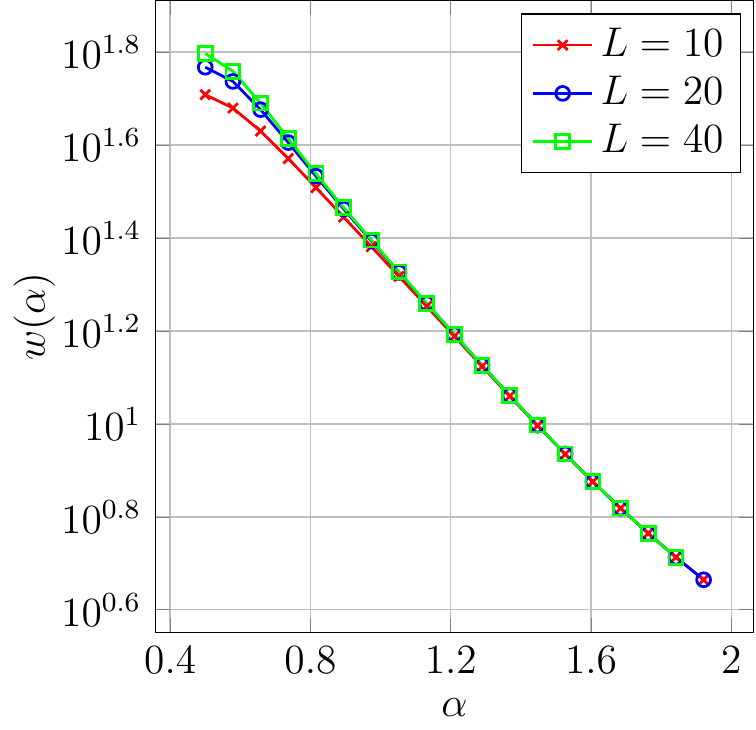}
    \caption{Interface width $w$.}
    \label{fig:comp_width_on_intervals}
  \end{subfigure}
  \caption{Comparing results on intervals~$[-L,L]$ with $L=10,20,40$ and $\eps=0.1$.}
  \label{fig:dependence_on_domain}
\end{figure}

We illustrate the metastable behavior for solutions of the Allen-Cahn equation~\eqref{eq:AC} with well--prepared initial data~\eqref{eq:TL:simple} with $L=10$, $\eps=0.01$ and $\alpha=0.9$ in Figure~\ref{fig:metastability}. In this simulation, we can define a time--to--collision $\widetilde{T}_C$ as the first time that two interfaces collide, i.e. $\widetilde{T}_C$ is the time at which the solution becomes non-negative.
This happens at some time between $t=3671.53$ and $t=3694.24$, see Figure~\ref{fig:metastability}.
We studied the evolution of interfaces also via the ODE system~\eqref{ODE:centers} which governs the evolution of the centers of interfaces. 
There the time--to--collision $T_C$ is defined as the first time that two centers collide, see~\eqref{cond:TC}. 
Using the parameters of our simulation, $T_C$ as given in~\eqref{eq:TC} is approximately $3587.10$, where we used $d_0=2$ as in the simulation and computed $\gamma=1.28550$ as in Appendix~\ref{subsec:layer_solution_on_R}. Comparing $\widetilde{T}_C$ with $T_C$ shows that the collision of interfaces takes place at a little later than expected.

\begin{figure}[h]
  \begin{subfigure}[b]{0.5\textwidth}
    \includegraphics[width=\textwidth]{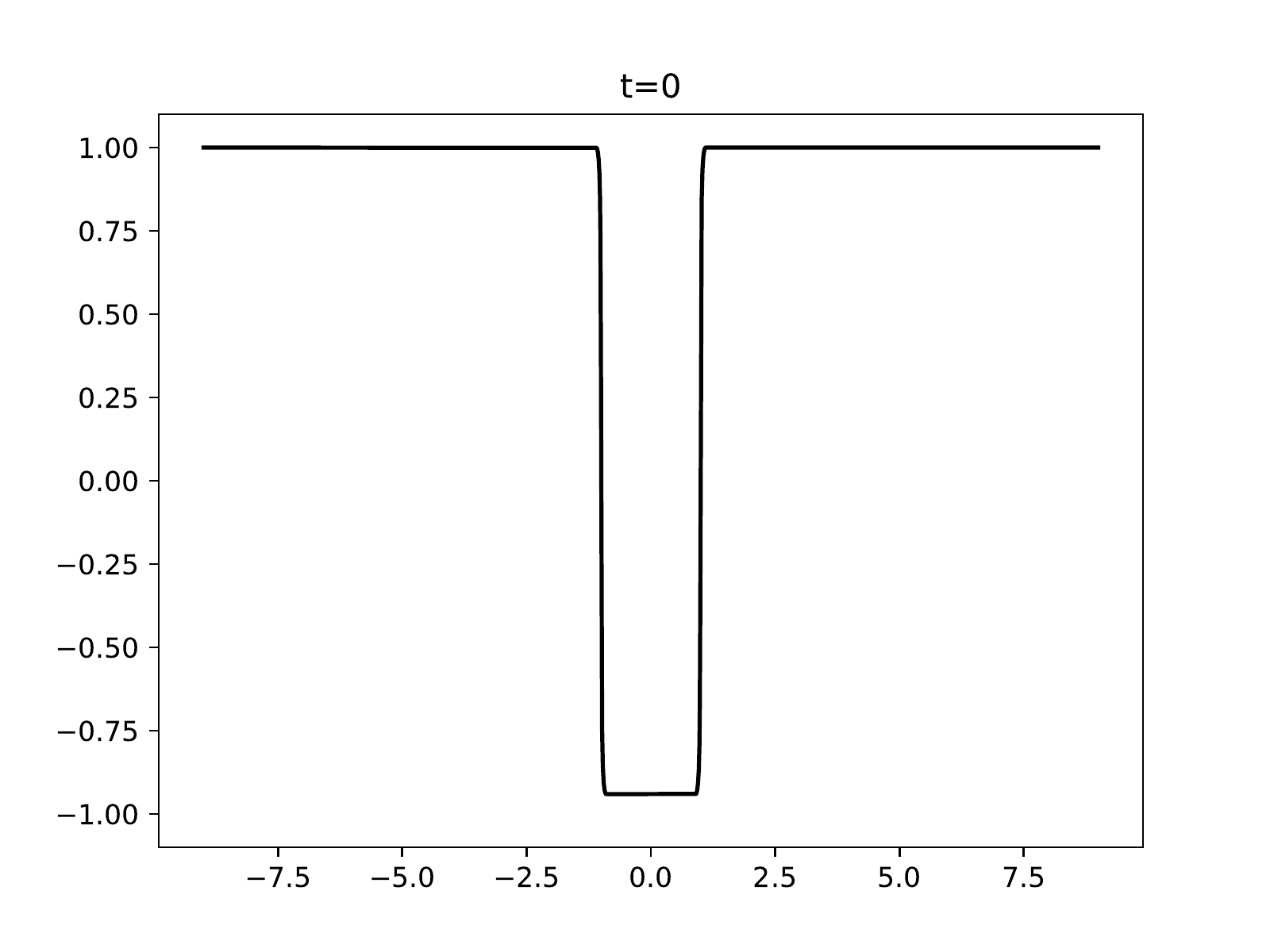}
    \caption{solution at time $t=0$.}
  \end{subfigure}\qquad
  \begin{subfigure}[b]{0.5\textwidth}
    \includegraphics[width=\textwidth]{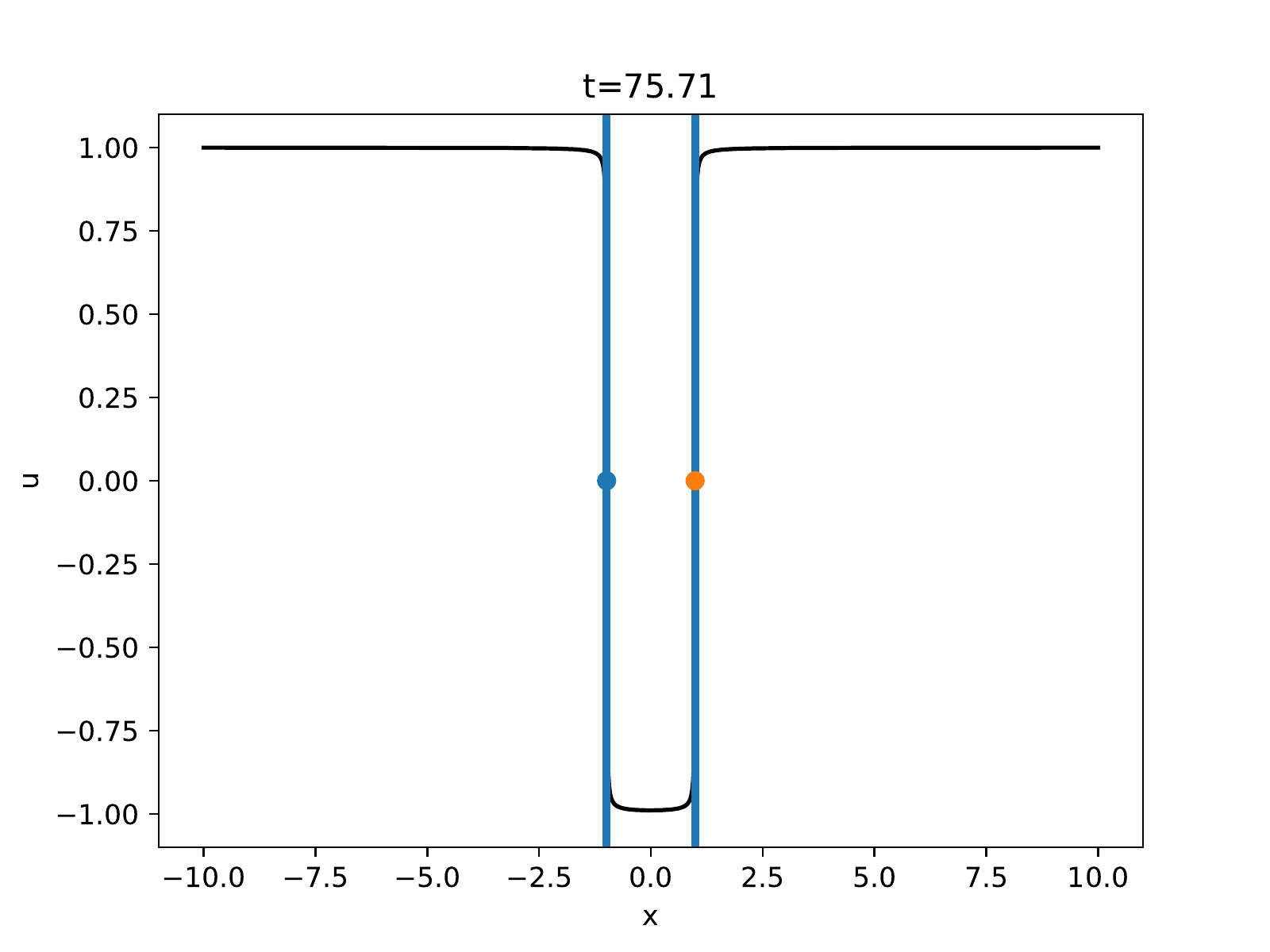}
    \caption{solution at time $t=75.71$.}
  \end{subfigure}\qquad 
  \begin{subfigure}[b]{0.5\textwidth}
    \includegraphics[width=\textwidth]{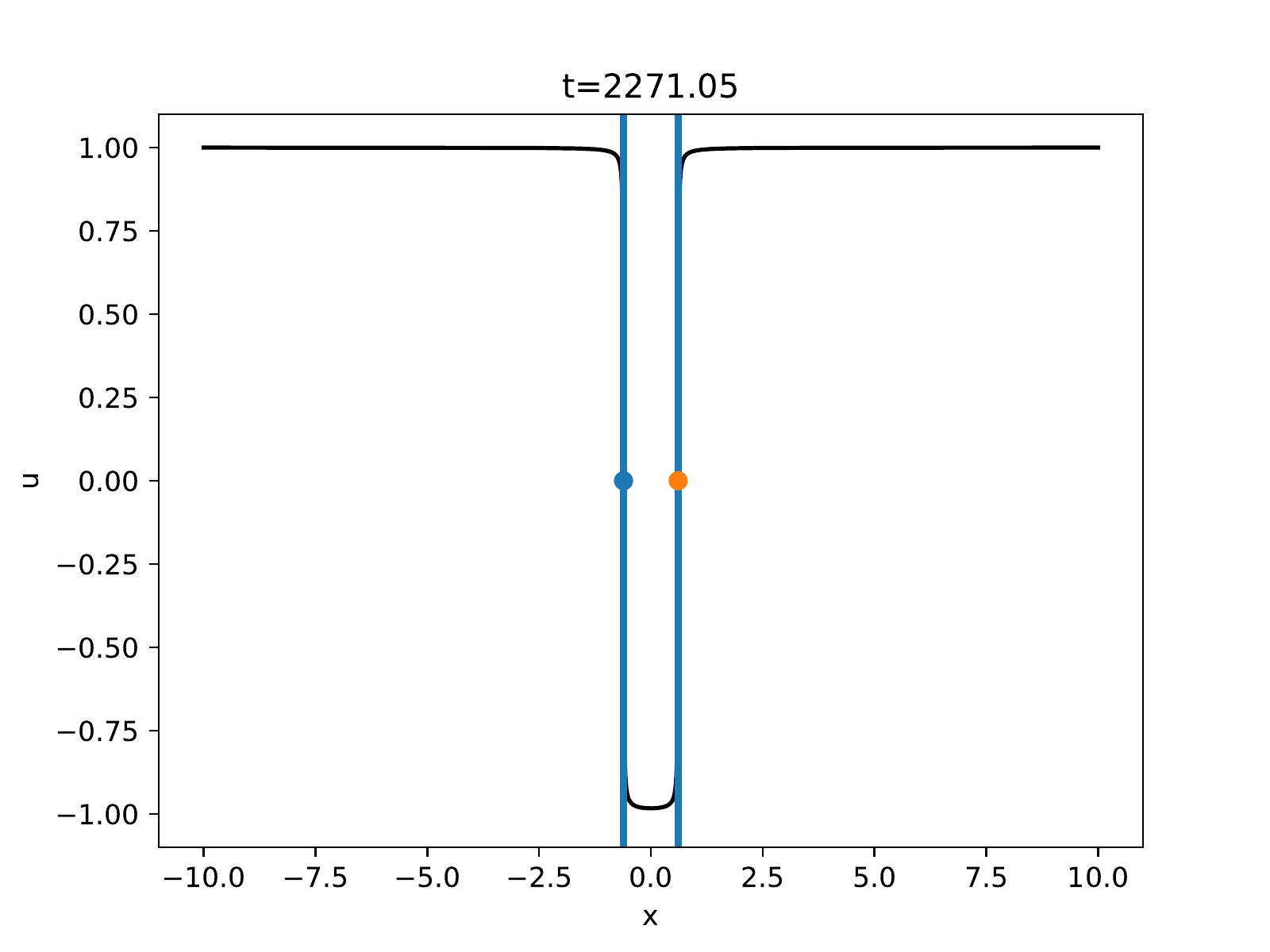} 
    \caption{solution at time $t=2271.05$.}
  \end{subfigure} 
  \begin{subfigure}[b]{0.5\textwidth}
    \includegraphics[width=\textwidth]{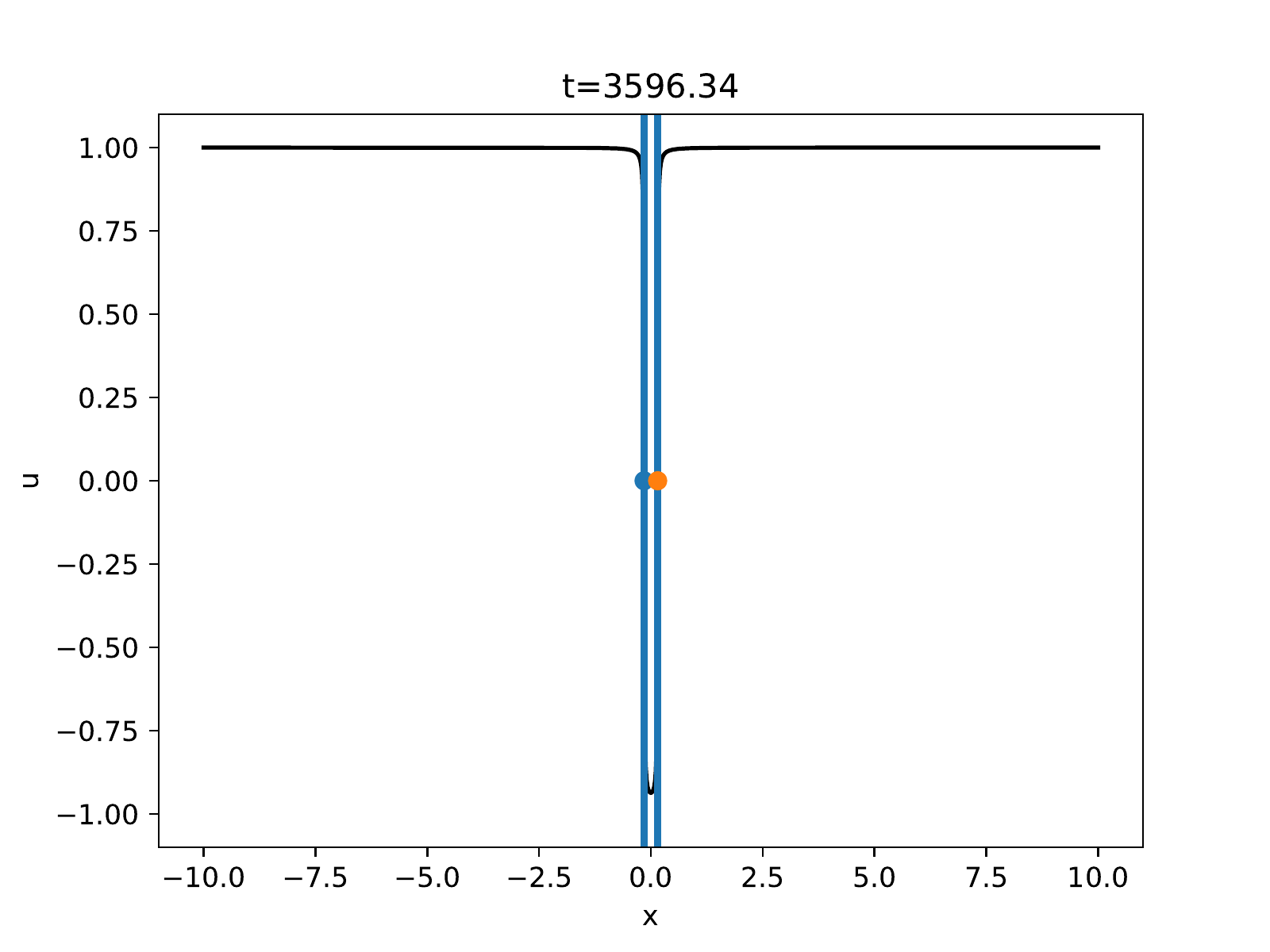}
    \caption{solution at time $t=3596.34$.}
  \end{subfigure}\\\vspace{2mm}
  \begin{subfigure}[b]{0.5\textwidth}
    \includegraphics[width=\textwidth]{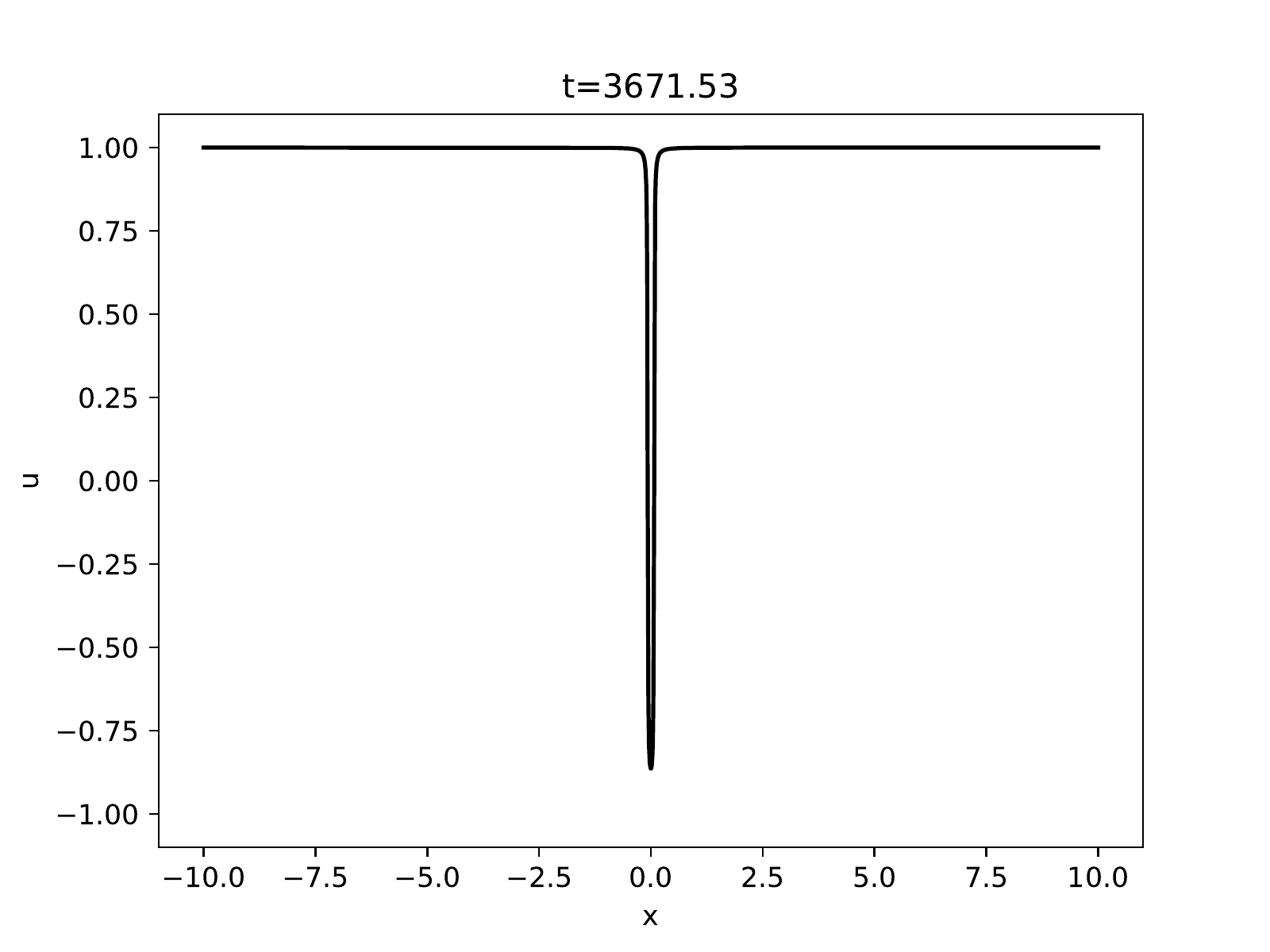}
    \caption{solution at time $t=3671.53$.}
  \end{subfigure}
  \begin{subfigure}[b]{0.5\textwidth}
    \includegraphics[width=\textwidth]{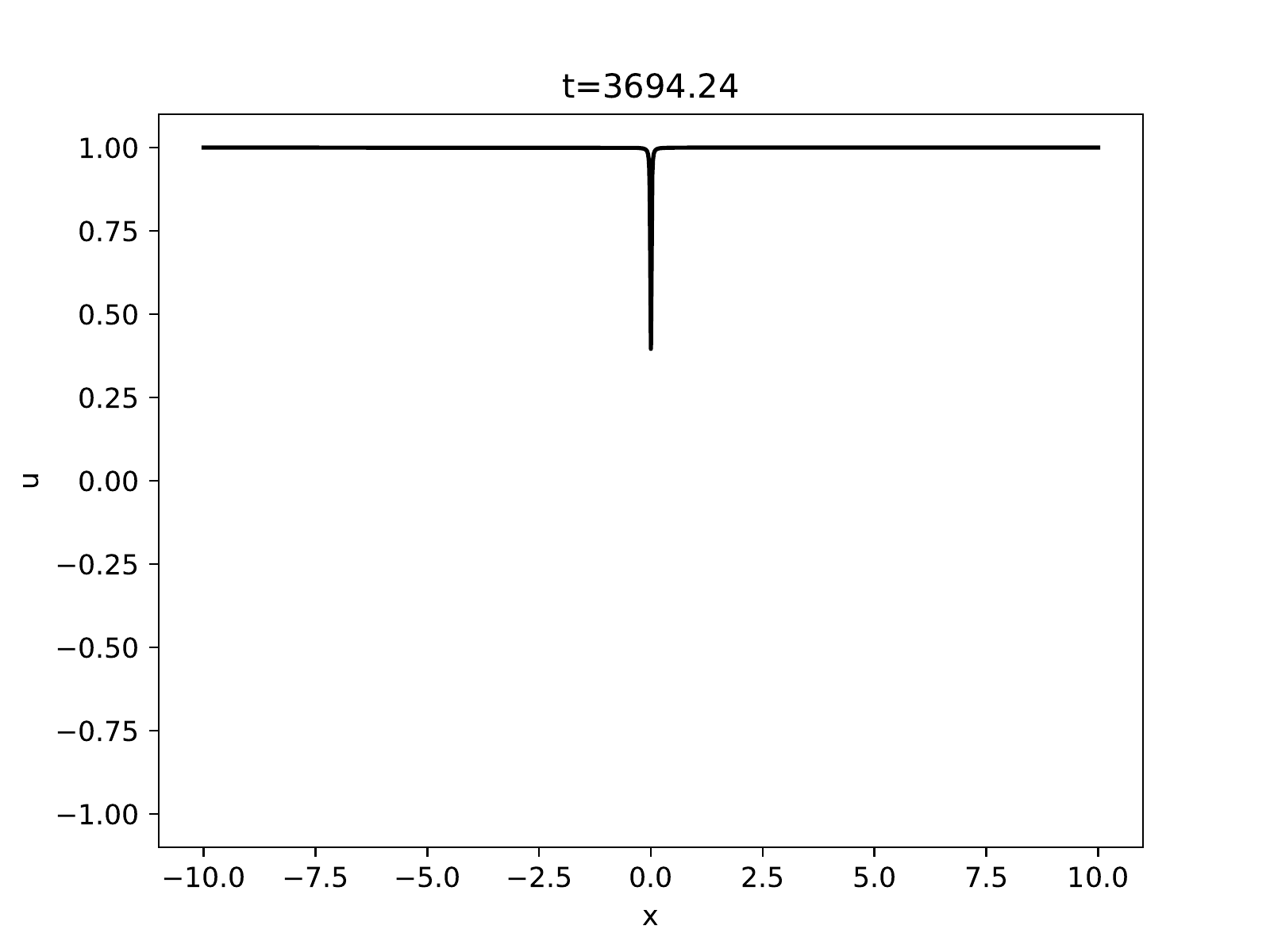}
    \caption{solution at time $t=3694.24$.}
  \end{subfigure}
  \caption{Evolution of solution for Allen-Cahn equation~\eqref{eq:AC} with well--prepared initial data~\eqref{eq:TL:simple} with $L=10$, $\eps=0.01$ and $\alpha=0.9$.}
  \label{fig:metastability}
\end{figure}

Next, we will first discuss the interface speed $s=s(\eps,\alpha)$ and, then, the interface width~$w=w(\eps,\alpha)$.

\subsection{Interface speed~$s=s(\eps,\alpha)$}
\label{sec:interface_speed}

Due to the analysis of the evolution of interfaces of~\eqref{eq:AC} on the real line, the (centers of) interfaces move according to the ODE system~\eqref{ODE:centers}.
In the case of two interfaces with centers $x_1$ and $x_2$, the interface speed satisfies approximately
\begin{equation} \label{speed_R}
 \begin{split}
 s_{\R} (\eps,\alpha)
  &\approx \Big| \frac{\eps^{1+\alpha}}{\alpha} 4 \frac{2^{\alpha} \Gamma((1+\alpha)/2)}{\sqrt{\pi} |\Gamma(-\alpha/2)|} \gamma \frac{x_1 -x_2}{|x_1 -x_2|^{1+\alpha}} \Big| \\
  &= \eps^{1+\alpha}\ \frac{4}{\alpha} \frac{2^{\alpha} \Gamma((1+\alpha)/2)}{\sqrt{\pi} |\Gamma(-\alpha/2)|}\ \gamma\ \frac{1}{|x_1 -x_2|^{\alpha}} \,.
 \end{split} 
\end{equation}
where $\gamma := (\int_\R (v'(x))^2 \txtd x)^{-1}$ and $v$ is a basic layer solution of~\eqref{eq:BLS:1} satisfying $\lim_{x\pm \infty} v(x) =\pm 1$ and $v(0) =0$.

In order to cross-validate this analysis, we measured the speed numerically. Renormalizing according to our previous considerations, we compute a renormalized speed
\begin{align*}
  \widehat{s}(\eps,\alpha)
  &:=\frac{ C_{\alpha} |x_1-x_2|^{\alpha}}{\eps^{1+\alpha} } s(\eps,\alpha).
\end{align*}
with the normalization factor 
\[
 C_\alpha
  :=\Bigg( \frac{4}{\alpha} \frac{2^{\alpha} \Gamma((1+\alpha)/2)}{\sqrt{\pi} |\Gamma(-\alpha/2)|}\Bigg)^{-1}
  =\frac{\Gamma(1-\alpha)}{2\sec(\alpha \frac{\pi}{2})} \,.
\]
If the relationship~\eqref{speed_R} holds true, we expect that the renormalized speed then behaves like the factor~$\gamma$.
In Figure~\ref{fig:normalized_speed}, we have plotted the relationship $\alpha \mapsto \widehat{s}(\eps,\alpha)$ for different values of~$\eps$. 
We first note, that all the curves (roughly) correspond to each other which hints that the
$\eps$--dependence is properly captured by the asymptotics. 
Comparing the resulting curve to the numerically computed value of~$\gamma$ then
confirms the overall relationship.

\begin{figure}
  \centering
  \includeTikzOrEps{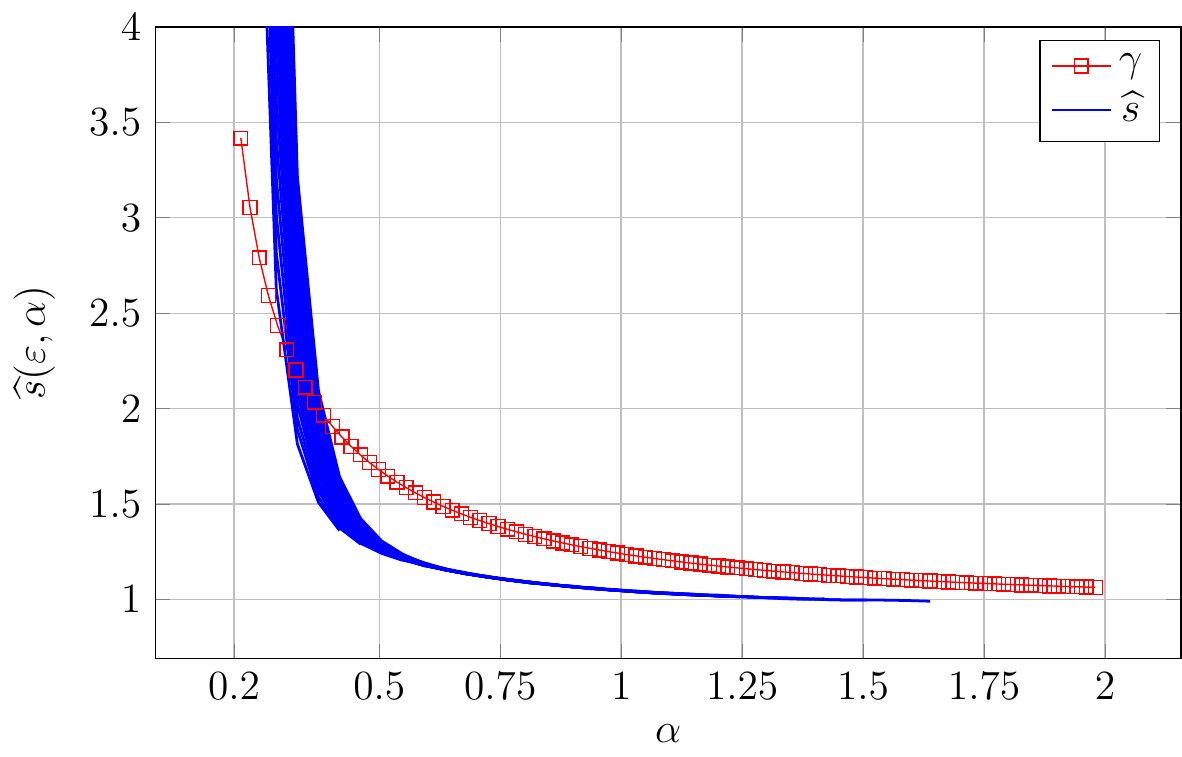}
  \caption{Comparison of the renormalized speed~$\widehat{s}(\eps,\alpha)$ to $\gamma$.
  Blue: different lines represent different values of $\varepsilon \in (0.002,0.05)$.}
  \label{fig:normalized_speed}
\end{figure}

\subsection{Time--to--collapse $t_{col}$}
The previous computation of the interface speed still relied on locating the interface and measuring its speed using a number of sample points. This might cause some hard to account for inaccuracies, resulting in small discrepancies to the predicted behavior and the slightly noisy look of Figure~\ref{fig:normalized_speed}.
In order to arrive at a second, more robust result, we also looked at the
time it takes for the interfaces to collide and finally annihilate, what we call the \textit{time--to--collapse}~$t_{col}$ for solutions of PDE~\eqref{eq:AC}.
We compute it by comparing our numerical solution to the constant function
taking the value $1$ everywhere. 
If the difference, as measured in the $L^2$-norm, drops below $10^{-6}$ we record the current time and plot the analog to Figure~\ref{fig:normalized_speed}.

The time--to--collapse~$t_{col}$ for solutions of the Allen-Cahn equation~\eqref{eq:AC} with well--prepared initial data~\eqref{eq:TL:simple} is approximately/greater than the time--to--collision~$T_C$ in the associated ODE system~\eqref{ODE:centers} (up to some additive terms, see Theorems~\ref{thm:TC} and~\ref{thm:Teps}):
\begin{align}
  t_{col}(\eps,\alpha)
  \geq \bigg(\frac{d_0}{\eps}\bigg)^{1+\alpha} \frac{C_\alpha}{2 (1+\alpha) \gamma}.
\end{align}
Computing the modified time--to--collapse as 
\begin{equation*}
  \widehat{t}_{col}
  :=\bigg(\frac{\eps}{d_0}\bigg)^{1+\alpha} \frac{2 (1+\alpha)}{C_\alpha} t_{col} \,,
\end{equation*}
we expect that $\widehat{t}_{col} \geq \gamma^{-1}$ but to be of comparable size.
Figure~\ref{fig:normalized_time} confirms the estimate on~$\widehat{t}_{col}$.

\begin{figure}
  \centering
  \includeTikzOrEps{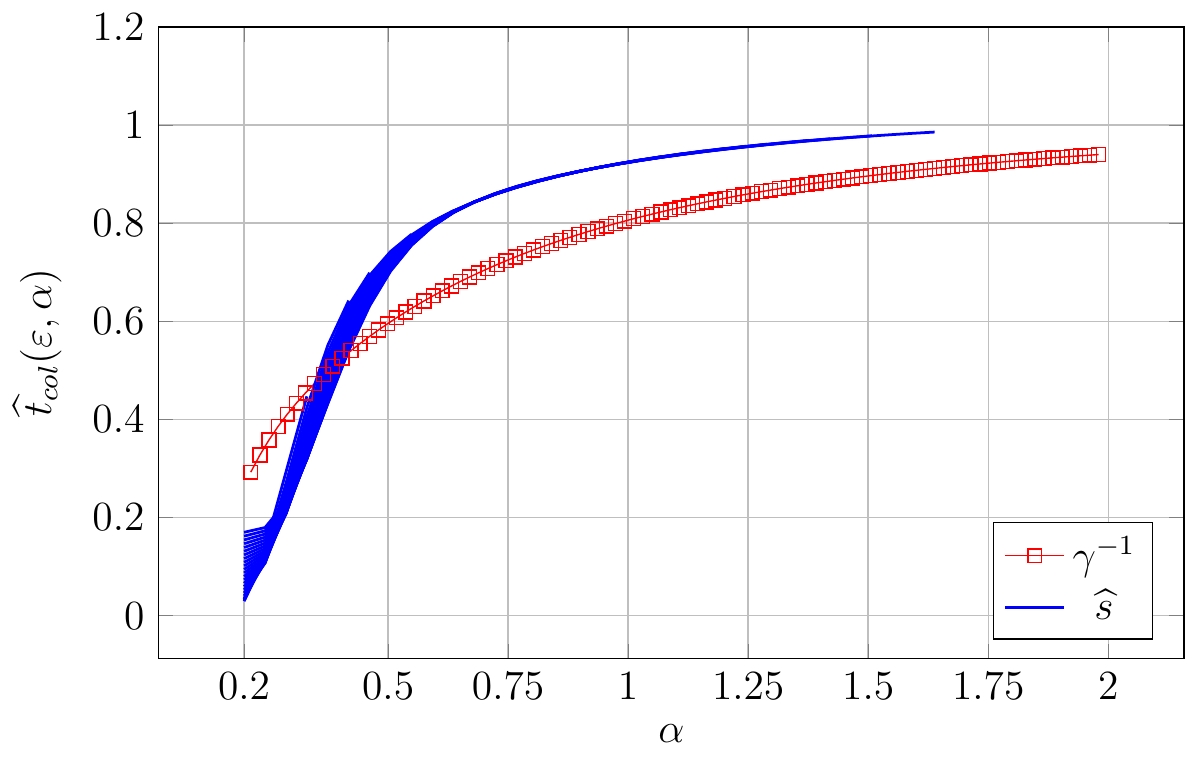}
  \caption{ Comparison of the renormalized time--to--collapse~$\widehat{t}_{col}$ to $1/\gamma$.
    Blue: different lines represent different values of $\varepsilon  \in (0.002,0.05)$.}
  \label{fig:normalized_time}
\end{figure}

\subsection{Interface width $w=w(\eps,\alpha)$}
Next, we consider the width of individual interfaces. 
For each fixed $\alpha$, we can plot the interface width as a function of $\eps$ in a log-log plot as shown in Figure~\ref{fig:interface_width}. 

\begin{figure}
  \center
  \begin{subfigure}[b]{0.46\textwidth}
    \includeTikzOrEps{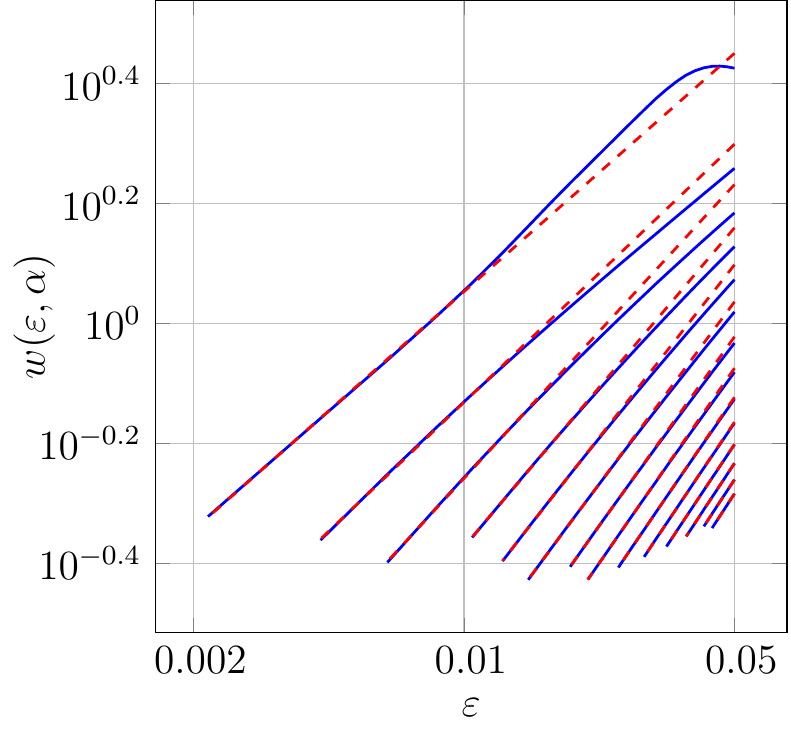}    
  \end{subfigure}\quad
  \begin{subfigure}[b]{0.46\textwidth}
    \includeTikzOrEps{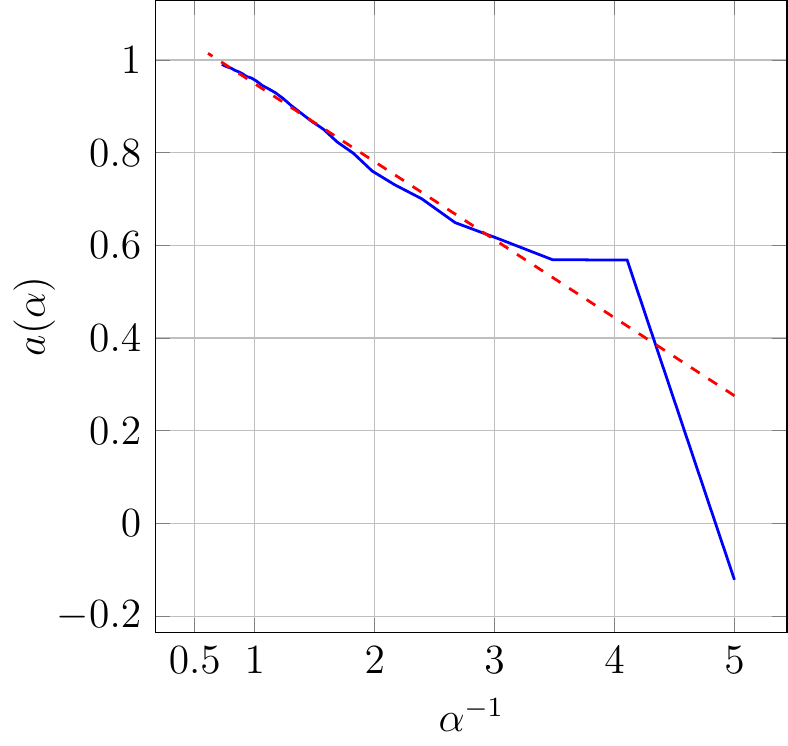}
  \end{subfigure}
  \caption{Behavior of the interface width (hp-BDF2 method). Different lines represent different values of $\alpha$.
  Blue line: actual data; black dashed line: power law fit $w(\eps,\alpha)\approx \eps^{a(\alpha)} b(\alpha)$.}
  \label{fig:interface_width}
\end{figure}

We see that the behavior with respect to $\eps$ can be approximated well by a power law of the form $w(\eps,\alpha) \approx \eps^{a(\alpha)} b(\alpha)$. In the second subplot of Figure~\ref{fig:interface_width}, we plotted the behavior of the coefficient $a(\alpha)$ when varying $\alpha$ in $(0.2,2)$. Namely, we chose a uniform grid of 40 points between $0.2$ and $1.9$. We see that the behavior especially in the regime of larger $\alpha$, can be well modeled by $a(\alpha) \approx \kappa_1 \alpha^{-1}+\kappa_2$. 
Empirically we determined the constants in the fit as $\kappa_1:=-0.168298$ and $\kappa_2:= 1.11709$.
Overall we see the empirical law 
\begin{align*}
  w(\eps,\alpha)\approx b(\alpha) \eps^{-0.168298 \alpha^{-1} +  1.11709},
\end{align*}
which is included in Figure~\ref{fig:interface_width} as the dotted black line. In order to get an estimate on $b(\alpha)$, we again rescale $w$ to get
\begin{align}
  \label{eq:definition_what}
 \widehat{w}(\eps,\alpha):=w(\eps,\alpha) \eps^{0.168298 \alpha^{-1} - 1.11709} .
\end{align}
We plot the behavior of $\widehat{w}$ in Figure~\ref{fig:modified_width}.
We observe that, while the $\eps$-dependence is well captured for larger values
of $\alpha$, the behavior for small $\alpha$ is more erratic.

\begin{figure}
  \begin{subfigure}{0.5\textwidth}
    \includeTikzOrEps{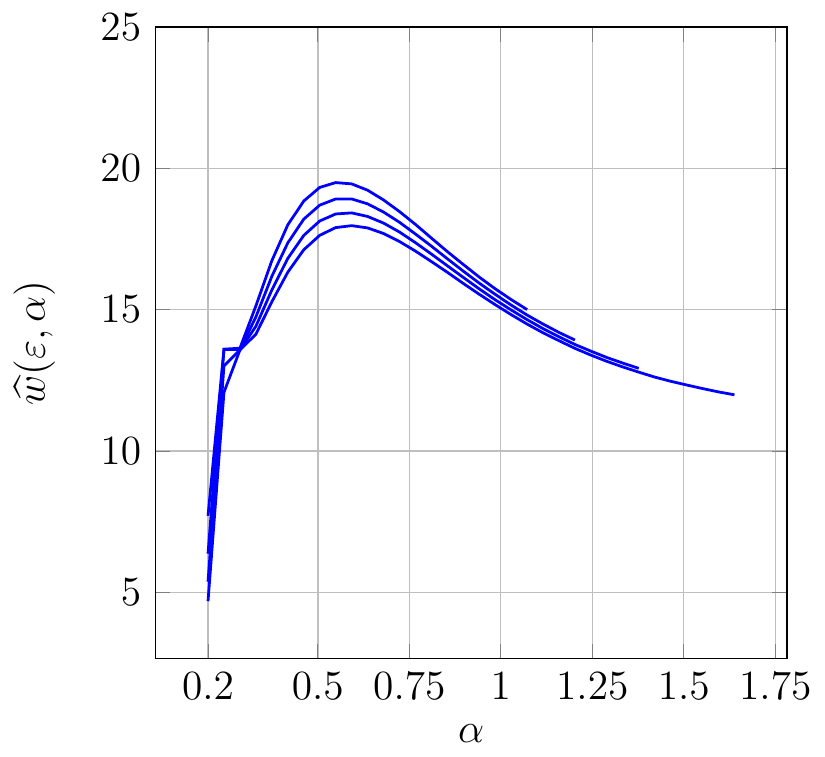}
    \caption{Behavior of $\widehat{w}(\eps,\alpha)$ in~\eqref{eq:definition_what}.}
    \label{fig:modified_width_a}
  \end{subfigure}
  \begin{subfigure}{0.5\textwidth}
    \includeTikzOrEps{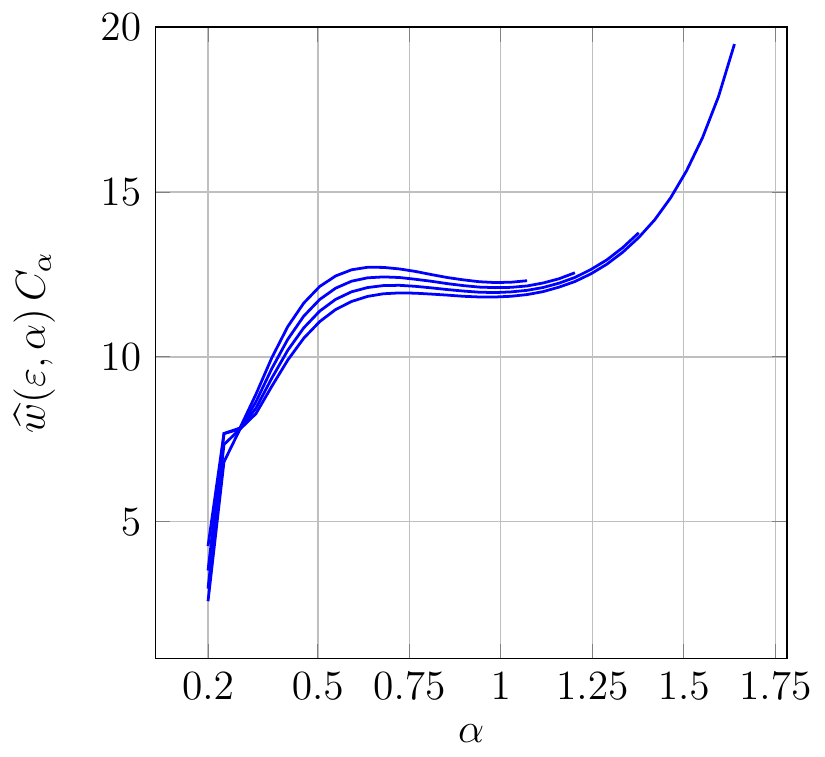}
    \caption{Same  as in~(a), but also correcting using $C_{\alpha}$.}
  \end{subfigure}
  \begin{subfigure}{0.5\textwidth}
    \includeTikzOrEps{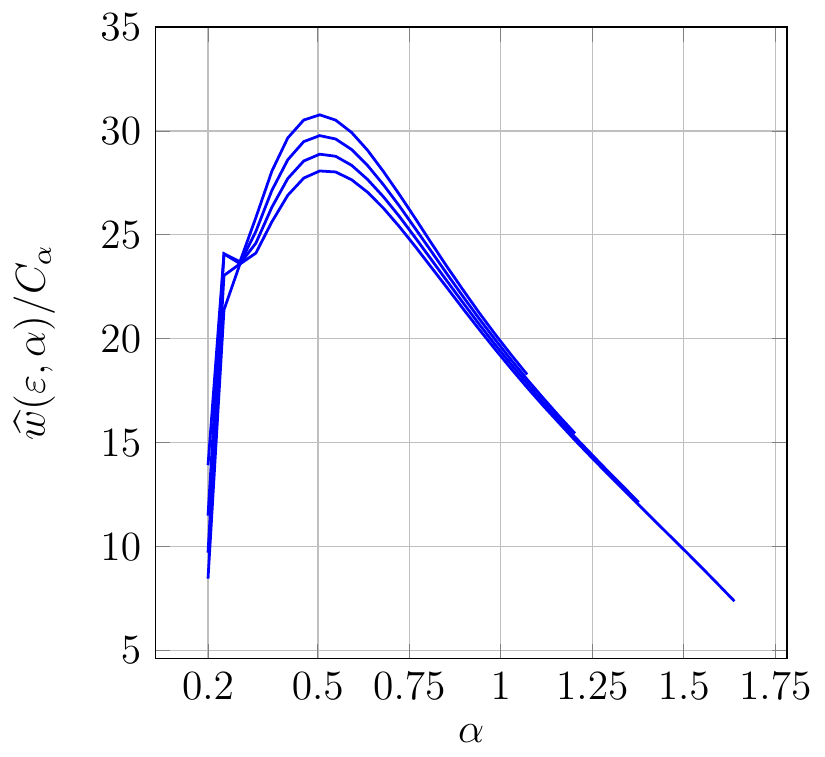}
    \caption{Same  as in~(a), but also correcting using $C_{\alpha}$ (alternative version).}
  \end{subfigure}

  \caption{Behavior of the modified interface width $\widehat{w}(\varepsilon,\alpha)$ with respect to $\alpha$.
    Different lines represent different values of $\varepsilon$.}
\label{fig:modified_width}
\end{figure}

\section{Summary and Outlook}
\label{sec:outlook}

Following the analysis of a fractional Allen-Cahn equation on~$\R$, we obtained formulas for the interface speed and time--to--collision and its dependence on $\eps$ and $\alpha$. In Section~\ref{sec:results} we checked the validity of these formulas for our model on a bounded interval with homogeneous Neumann boundary conditions. Moreover, we studied numerically the behavior of the interface width well before the two interfaces collide. Following the analysis by Patrizi and Valdinoci, we considered \textit{well--prepared initial data}, i.e., we started close to a metastable profile. Hence, we are now in the position to provide a good dynamical unterstanding of the fractional Allen-Cahn equation in its metastable regime. Yet, for the classical Allen-Cahn equation, it is well-known that there are four possible stages in the generation, propagation and annihilation of metastable patterns~\cite{Chen:2004}: Starting from a large class of oscillating initial data, there is phase separation, then the formation of a metastable pattern, next a slow motion of the metastable pattern and (a possible) annihilation of two interfaces. The latter stages may repeat until either a constant stable state or a single basic layer solution remains. Thus, we are still missing a more detailed analysis of the early stages of the generation of metastable patterns for the fractional Allen-Cahn equation, which remains a topic for future work.

Another natural question is to try to extend the approach in~\cite{CarrPego} to the case of a fractional Allen-Cahn equation, i.e., to re-consider the auxiliary BVP~\eqref{eq:BVP} as a problem in its own right. We may ask what happens if we replace the second-order derivative by the fractional Laplacian? Is this problem still somehow connected to metastability? Is the solution unique if we require positivity in the interior? If so, what is its shape? If it has a peak near zero, what is the scaling of the spike layer as a function of $\ell$, $\eps$ and $\alpha$? These questions all remain to be dealt with in future work.

Another important direction is to consider other potentials $F$. In this work we considered a balanced potential, i.e. $F(-1)=F(1)$. To fully characterize the dynamics of the fractional Allen-Cahn equation for general unbalanced potentials is an open question, but see~\cite[\S 5]{NecNepGol:2008} and \cite{AchKue:2015, CabSir:2015} as examples for encouraging results towards this goal.

\appendix

\section{Dependence on the size of the domain}

Since all our analytical insights are dependent on working on the full space $\R$, whereas our numerics is restricted to a finite interval, it is natural to investigate the dependence of our quantities of interest on the size of the considered interval $(-L,L)$. 
As can be seen in Figure~\ref{fig:dependence_on_domain}, neither the (modified) speed nor the width of the interface seems to depend strongly on the size of the domain. Especially for larger domains, the results become indistinguishable.
%

\subsection{Basic layer solution on $\R$}
\label{subsec:layer_solution_on_R}
The layer solutions of~\eqref{eq:BLS:1}--\eqref{eq:BLS:2} are -- in fact -- the stationary traveling wave solutions of~\eqref{eq:AC}, which are unique up to translations (see~\cite{CabSir:2015}) and locally asymptotically stable, see~\cite{AchKue:2015, MaNiuWang:2019}.
Therefore we aim to approximate the layer solution as the stationary solution of~\eqref{eq:AC}.
We use again the numerical scheme on intervals $[-L,L]$ as presented in Section~\ref{sec:numerics}.
Starting our numerical scheme in Section~\ref{sec:numerics} with $\eps=1$ and an initial datum~$u_0$ given as the linear function connecting $-1$ and $1$, we let the simulation run until a stationary state is reached. Finally we compute 
\[ \int_{-\infty}^{\infty}{|v'|^2 \txtd x}
 = \int_{-\infty}^{-L}{|v'|^2 \txtd x} +\int_{-L}^{L}{|v'|^2 \txtd x} +\int_{L}^{\infty}{|v'|^2 \txtd x}\,.
\]
The term $\int_{-L}^{L}{|v'|^2 \txtd x}$ is computed from our numerical solution. 
Whereas for the first and third summand we use that the layer solution is an odd function (see \cite[Thm. 2.4]{CabSir:2015} and its tail behaves as 
\[
 v(x) \approx 1 +p x^{-\alpha} \quad\text{for } x\gg L \quad\text{where } p = -\frac{\sec(\alpha\pi/2)}{2\Gamma(1-\alpha)}\,,
\]
see \cite[Eq. (23)]{NecNepGol:2008} and \cite[Thm. 2.7]{CabSir:2015}.
Thus,
\begin{align*}
 \int_{-\infty}^{-L}{|v'|^2 \txtd x}
  &=\int_{L}^{\infty}{|v'|^2 \txtd x} \\
  &=\int_{L}^{\infty}{p^2 \alpha^2 x^{-2\alpha -2} \txtd x} \\
  &=p^2 \alpha^2 \frac{-1}{2\alpha +1} x^{-2\alpha -1} \Big|_{L}^{\infty} \\
  &=p^2 \alpha^2 \frac{1}{2\alpha +1} L^{-2\alpha -1} \,.
\end{align*}
Summing up, we deduce 
\[ \int_{-\infty}^{\infty}{|v'|^2 \txtd x}
 = \int_{-L}^{L}{|v'|^2 \txtd x} +2 p^2 \alpha^2 \frac{1}{2\alpha +1} L^{-2\alpha -1}\,.
\]
We again chose different values of $L$ in order to investigate the influence of the choice of bounded domain, but observe that for larger intervals, the influence of the cutoff becomes negligible.

\begin{figure}
  \begin{subfigure}{0.5\linewidth}    
    \includeTikzOrEps{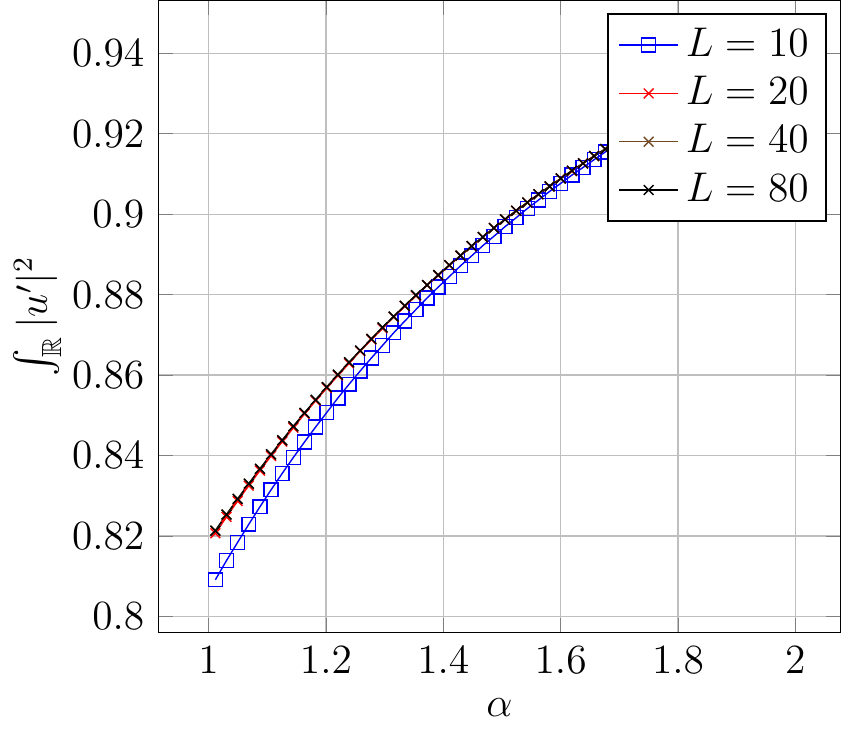}
  \end{subfigure}
  \begin{subfigure}{0.5\linewidth}
    \includeTikzOrEps{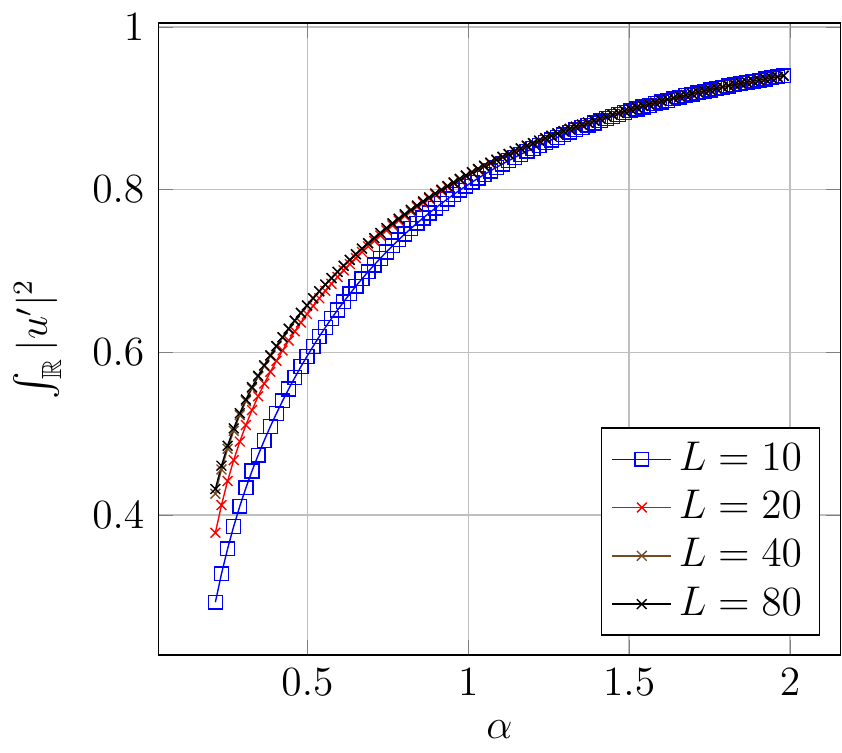}
  \end{subfigure}
  \caption{$H^1$-seminorm $\int_{-\infty}^{\infty}{|v'|^2 \txtd x}$ of a basic layer solution~$v$ of~\eqref{eq:BLS:1}--\eqref{eq:BLS:2}.}
\end{figure}

\textbf{Acknowledgments:} The authors (FA, JMM, AR) gratefully acknowledge financial support by the Austrian Science Fund (FWF) through
the research program ``Taming complexity in partial differential systems'' (grant SFB F65).
AR has in addition been funded through project P29197-N32. CK gratefully 
acknowledges support via a Lichtenberg Professorship of the 
VolkswagenStiftung as well as partial support of the SFB/TR 109 
Discretization in Geometry and Dynamics funded by the German Science 
Foundation (DFG).

\bibliographystyle{plain}
\bibliography{references}






\end{document}